\documentclass[
11pt,  reqno]{amsart}
\usepackage[margin=1.2in,marginparwidth=1.5cm, marginparsep=0.5cm]{geometry}

\usepackage{amsmath,amssymb,amscd,amsthm,amsxtra, esint}

\usepackage{booktabs} 
\usepackage{microtype}
\usepackage{amssymb}
\usepackage{mathrsfs}

\usepackage{upgreek} 

\usepackage{color}
\usepackage[implicit=true]{hyperref}

\usepackage[colorinlistoftodos,prependcaption,color=yellow,textsize=tiny,textwidth=2cm]{todonotes}

\usepackage{xsavebox}

\usepackage{lipsum}

\usepackage{cases}

\allowdisplaybreaks[2]

\definecolor{gr}{rgb}   {0.,   0.69,   0.23 }
\definecolor{bl}{rgb}   {0.,   0.5,   1. }
\definecolor{mg}{rgb}   {0.85,  0.,    0.85}
\definecolor{yl}{rgb}   {0.8,  0.7,   0.}
\definecolor{or}{rgb}  {0.7,0.2,0.2}

\newcommand{\wick}[1]{\wcol#1\wcol}

\newcommand{\VVert}[1]{{\left\vert\kern-0.25ex \left\vert\kern-0.25ex \left \vert #1 
    \right\vert\kern-0.25ex \right\vert\kern-0.25ex \right\vert}} 

\newcommand{\1}{\hspace{0.2mm}\textup{I}\hspace{0.2mm}}

\newtheorem{theorem}{Theorem} [section]

\newtheorem{lemma}[theorem]{Lemma}
\newtheorem{proposition}[theorem]{Proposition}
\newtheorem{remark}[theorem]{Remark}

\newtheorem{definition}[theorem]{Definition}

\DeclareMathOperator*{\supp}{supp}

\DeclareMathOperator{\Law}{Law}

\DeclareMathSymbol{\wcol}{\mathord}{operators}{"3A}

\newcommand{\W}{\mathcal{W}}

\newcommand{\dr}{\theta}

\newcommand{\Ha}{\mathbb{H}_a}
\newcommand{\noi}{\noindent}
\newcommand{\Z}{\mathbb{Z}}
\newcommand{\R}{\mathbb{R}}
\newcommand{\C}{\mathbb{C}}
\newcommand{\T}{\mathbb{T}}
\newcommand{\N}{\mathbb{N}}

\let\Re=\undefined\DeclareMathOperator*{\Re}{Re}

\let\P= \undefined
\newcommand{\P}{\mathbf{P}}
\newcommand{\PP}{\mathbb{P}}

\newcommand{\F}{\mathcal{F}}

\newcommand{\nb}{\nabla}

\newcommand{\dl}{\delta}

\newcommand{\eps}{\varepsilon}

\newcommand{\g}{\gamma}
\newcommand{\G}{\Gamma}
\newcommand{\ld}{\lambda}

\newcommand{\s}{\sigma}

\newcommand{\wt}{\widetilde}
\newcommand{\cj}{\overline}
\newcommand{\dx}{\partial_x}
\newcommand{\dt}{\partial_t}

\newcommand{\jb}[1]
{\langle #1 \rangle}

\renewcommand{\H}{\mathcal{H}}

\newcommand{\les}{\lesssim}
\newcommand{\ges}{\gtrsim}

\newcommand{\ind}{\mathbf 1}

\newcommand{\E}{\mathbb{E}}

\renewcommand{\o}{\omega}
\renewcommand{\O}{\Omega}

\numberwithin{equation}{section}
\numberwithin{theorem}{section}

\newtheorem*{ackno}{Acknowledgements}

\newcommand{\M}{\mathcal{M}}

\usepackage{tikz}

\usetikzlibrary{shapes.misc}
\usetikzlibrary{shapes.symbols}
\usetikzlibrary{shapes.geometric}
\tikzset{
	dot/.style={circle,fill=black,draw=black,inner sep=1pt,minimum size=0.5mm},
	>=stealth,
	}
\tikzset{
	ddot/.style={circle,fill=white,draw=black,inner sep=2pt,minimum size=0.8mm},
	>=stealth,
	}

\tikzset{decision/.style={ 
        draw,
        diamond,
        aspect=1.5
    }}

\tikzset{dia2/.style
={diamond,fill=white,draw=black,inner sep=0pt,minimum size=1mm},
	>=stealth,
	}

\tikzset{dia/.style
={star,fill=black,draw=black,inner sep=0pt,minimum size=1mm},
	>=stealth,
	}

\colorlet{symbols}{black}
\colorlet{testcolor}{green!60!black}

\def\1{\mathbf{{1}}}

\usetikzlibrary{shapes.misc}
\usetikzlibrary{shapes.symbols}
\usetikzlibrary{snakes}
\usetikzlibrary{decorations}
\usetikzlibrary{decorations.markings}
\definecolor{dblue}{rgb}{0.1, 0.1, 0.9}

\tikzset{
	root/.style={circle,fill=testcolor,inner sep=0pt, minimum size=2mm},		
	dot/.style={circle,fill=black,draw=black, solid,inner sep=0pt,minimum size=0.75mm},
	bdot/.style={circle,fill=blue,draw=dblue, solid,inner sep=0pt,minimum size=0.75mm},
		}
\colorlet{symbols}{blue!90!black}

\makeatletter
\def\DeclareSymbol#1#2#3{\expandafter\gdef\csname MH@symb@#1\endcsname{\tikz[baseline=#2,scale=0.15]{#3}}%
\expandafter\gdef\csname MH@symb@#1s\endcsname{\scalebox{0.6}{\tikz[baseline=#2,scale=0.15]{#3}}}}
\def\<#1>{\csname MH@symb@#1\endcsname}
\makeatother

\DeclareSymbol{sunset}{-3}{\draw (0,0) circle [x radius = 1.5, y radius = 1]; \draw (-1.5,0) node[dot] {} -- (1.5,0) node[dot] {};}
\DeclareSymbol{tadpole}{-3}{\draw (0,0) circle [x radius = 0.75, y radius = 1]; \draw (0,-1) node[dot] {};}

\DeclareSymbol{1}{0}{\draw[white] (-.4,0) -- (.4,0); \draw (0,0)  -- (0,1.2) node[dot] {};}
\DeclareSymbol{2}{0}{\draw (-0.5,1.2) node[dot] {} -- (0,0) -- (0.5,1.2) node[dot] {};}
\DeclareSymbol{3}{0}{\draw (0,0) -- (0,1.2) node[dot] {}; \draw (-.7,1) node[dot] {} -- (0,0) -- (.7,1) node[dot] {};}
\DeclareSymbol{4}{0}{\draw (-0.36,1.2) node[dot] {} -- (0,0) -- (0.36,1.2) node[dot] {}; \draw (-1,1) node[dot] {} -- (0,0) -- (1,1) node[dot] {};}
\DeclareSymbol{30}{-3}{\draw (0,0) -- (0,-1); \draw (0,0) -- (0,1.2) node[dot] {}; \draw (-.7,1) node[dot] {} -- (0,0) -- (.7,1) node[dot] {};}
\DeclareSymbol{31}{-3}{\draw (0,0) -- (0,-1) -- (1,0) node[dot] {}; \draw (0,0) -- (0,1.2) node[dot] {}; \draw (-.7,1) node[dot] {} -- (0,0) -- (.7,1) node[dot] {};}
\DeclareSymbol{32}{-3}{\draw (0,0) -- (0,-1) -- (1,0) node[dot] {}; \draw (0,0) -- (0,-1) -- (-1,0) node[dot] {}; \draw (0,0) -- (0,1.2) node[dot] {}; \draw (-.7,1) node[dot] {} -- (0,0) -- (.7,1) node[dot] {};}
\DeclareSymbol{20}{-3}{\draw (0,0) -- (0,-1);\draw (-.7,1) node[dot] {} -- (0,0) -- (.7,1) node[dot] {};}
\DeclareSymbol{22}{-3}{\draw (0,0.3) -- (0,-1) -- (1,0) node[dot] {}; \draw (0,0.3) -- (0,-1) -- (-1,0) node[dot] {};\draw (-.7,1) node[dot] {} -- (0,0.3) -- (.7,1) node[dot] {};}
\DeclareSymbol{202}{-3}{\draw (-.6625,0) -- (0,-1) -- (.6625,0)  {}; \draw (-1.1,1) node[dot] {} -- (-.6625,0) -- (-0.365,1) node[dot] {}; \draw (0.365,1) node[dot] {} -- (.6625,0) -- (1.1,1) node[dot] {};}

\makeatletter
\def\DeclareSymbol#1#2#3{\expandafter\gdef\csname MH@symb@#1\endcsname{\tikz[baseline=#2,scale=0.15]{#3}}}
\def\<#1>{\csname MH@symb@#1\endcsname}
\makeatother

\DeclareSymbol{X}{-2.4}{\node[dot] {};}
\DeclareSymbol{1}{0}{\draw[white] (-.4,0) -- (.4,0); \draw (0,0)  -- (0,1.2) node[dot] {};}
\DeclareSymbol{2}{0}{\draw (-0.5,1.2) node[dot] {} -- (0,0) -- (0.5,1.2) node[dot] {};}

\makeatletter
\def\DeclareSymbol#1#2#3{\expandafter\gdef\csname MH@symb@#1\endcsname{\tikz[baseline=#2,scale=0.15]{#3}}}
\def\<#1>{\csname MH@symb@#1\endcsname}
\makeatother

\begin{document}

\baselineskip = 14pt

\title[Large deviations for the Gross Pitaevskii Gibbs measure]
{Large deviation principles for the Gross Pitaevskii Gibbs measure at low temperature}

\author[L.~Packer, K.~Seong, and P.~Sosoe]
{Liam Packer, Kihoon Seong, and Philippe Sosoe}

\address{Liam Packer\\
Department of Mathematics\\
Cornell University\\ 
310 Malott Hall\\ 
Cornell University\\
Ithaca\\ New York 14853\\ 
USA }

\email{lp492@cornell.edu}

\address{Kihoon Seong\\
Department of Mathematics\\
Cornell University\\ 
310 Malott Hall\\ 
Cornell University\\
Ithaca\\ New York 14853\\ 
USA \\  and Simons Laufer Mathematical Sciences Institute (SLMath)\\
17 Gauss Way\\
Berkeley, CA 94720\\ USA}

\email{kihoonseong@msri.org}

\address{Philippe Sosoe\\
Department of Mathematics\\
Cornell University\\ 
310 Malott Hall\\ 
Cornell University\\
Ithaca\\ New York 14853\\ 
USA }

\email{ps934@cornell.edu}



\begin{abstract}

We prove the large deviation principle for the conditional Gibbs measure associated with the focusing Gross–Pitaevskii equation in the low-temperature regime. This conditional measure is of mixed type, being canonical in energy and microcanonical in particle number. In particular, our result extends the large deviation principle for the mixed ensemble studied by Ellis, Jordan, Otto, and Turkington to a more singular setting, where the interaction potential is unbounded and the conditional event involves diverging renormalization constants. As a consequence of the large deviation principle, the Gibbs measure concentrates along the soliton manifold in the low-temperature limit.







\end{abstract}

\maketitle

\tableofcontents

\section{Introduction}

\subsection{Large deviation principle (LDP) for the mixed  ensemble}

In this paper, we study a large deviation principle for the Gross-Pitaevskii Gibbs measure in the low temperature limit. This measure has the following formal expression
\begin{align}
d\rho_{\eps}^D(\phi)=Z_{\eps}^{-1}\exp\Big\{-\frac 1\eps H(\phi) \Big\} \dl(M(\phi)-D)
\prod_{x \in \R} d\phi(x),
\label{Gibbs1}
\end{align}

\noi
conditioned on the event $\{M(\phi)=D\}$, $D>0$, where $Z_{\eps}$\footnote{Here $Z_\eps$ denotes different normalizing constants that may differ from one line to line.} is the partition function, $\eps>0$ is the temperature parameter, and $\prod_{x\in \R} d\phi(x)$ denotes the (non-existent) Lebesgue measure on fields $\phi:\R \to \C$.  Here, $H$ is the Hamiltonian 
\begin{align}
H(\phi)=\frac{1}{2}\int_{\R} |\dx \phi|^2 dx+\frac 12 \int_{\R} |x|^2 |\phi|^2 dx-\frac \ld4 \int_{\R} |\phi|^4 dx,
\label{Ha0}
\end{align}

\noi
where $\ld>0$ is the coupling constant that measures the strength of the focusing interaction. For such a system, $|\phi|^2$  is interpreted as the particle density, and thus the total number of particles is given by
\begin{align}
M(\phi)=\int_{\R} |\phi|^2 dx.
\label{L2mass}
\end{align}

\noi 
The Gibbs measure \eqref{Gibbs1} is called a mixed ensemble, as it is canonical in the energy $H(\phi)$ and microcanonical in particle number $M(\phi)$, which is informally understood as a regular conditional distribution. As discussed in the previous works of Lebowitz, Rose, and Speer \cite{LRS}, McKean and Vaninsky \cite{McKVa}, Bourgain \cite{BO94, BO97a}, and Brydges and Slade \cite{BS}, such Gibbs ensembles are necessarily microcanonical in $M$, since the canonical Gibbs ensemble with respect to the Hamiltonian $H$ is not normalizable in the focusing case $\ld>0$, that is, the partition function $Z_\eps=\infty$ without the conditioning $M(\phi)=D$. 


The Gibbs ensemble \eqref{Gibbs1} is an invariant measure for the following Hamiltonian PDE, known as the Gross-Pitaevskii equation
\begin{align}
i \dt \psi +\dx^2 \psi- |x|^2 \psi -\ld |\psi|^2 \psi=0, \qquad (t,x) \in \R \times \R,
\label{GP}
\end{align}


\noi
where $\psi(t,x):\R \times \R \to \C$. This is a nonlinear Schr\"odinger equation with a confining harmonic potential $|x|^2$, which is known to model Bose-Einstein condensates.




In this paper, we establish a large deviation principle for the regular conditional probability measure \eqref{Gibbs1}:
\begin{align}
\lim_{\eps \to 0} \eps \log \rho^D_\eps\big(\{ \phi \in B \}\big)=-\inf_{\phi \in B} J^D(\phi)
\label{LDP0}
\end{align}

\noi
for appropriate sets $B \subset \supp \rho_\eps^D$, where $J^D$ is the large deviation rate function
\begin{equation}
\begin{split}
J^D(\phi)=
\begin{cases}
H(\phi)-\inf\limits_{M(\phi)=D} H(\phi) \quad &\text{if} \quad  \phi\in \mathcal{H}^1(\R) \quad \text{and} \quad M(\phi)=D \\
\infty \quad &\text{otherwise}.
\end{cases}
\end{split}
\label{JDrate}
\end{equation}

\noi
Here, $\mathcal{H}^1(\R)$ denotes the Cameron–Martin space associated with the Gaussian measure whose covariance operator is $\mathcal{L}:=(-\dx^2+|x|^2)^{-1}$. See Theorem \ref{THM:1} for the precise statement.

As a consequence of the large deviation principle in \eqref{LDP0}, the rate function $J^D$ penalizes deviations from minimal energy, favoring configurations near the lowest energy level.  Therefore,  we obtain a concentration phenomenon for the conditional Gibbs measure \eqref{Gibbs1}: for any $\dl>0$, there exists $c(\dl)>0$ such that 
\begin{align}
\rho_{\eps}^D\bigg( \Big\{ \inf_{Q\in \M^D} \|  \phi-Q    \|_{\mathcal{W}(\R) } \ge \dl   \Big\}  \bigg) \les e^{-\frac{c(\dl)}{\eps}}
\label{CE}
\end{align}

\noi 
as $\eps \to 0$, where $\M^D$ is the soliton manifold, that is, the family of minimizers of the Hamiltonian \eqref{Ha0} under the constraint $M(\phi)=D$, and $\mathcal{W}(\R)$ is a suitable norm defined below. Configurations $\phi$ far away from the minimizers of $J^D$ are therefore increasingly unlikely at low temperatures (with exponentially vanishing probability) due to their high cost $J^D(\phi)\geq c(\delta) > 0$. This shows that ground state solitary waves are the most probable macroscopic states. Since the Gibbs measure $\rho^D_\eps$ in \eqref{Gibbs1} is invariant under the flow of \eqref{GP}, high-probability events under $\rho^D_\eps$ capture the long-time dynamics.  Therefore, the concentration estimate \eqref{CE} implies that at low temperature, the long-term behavior of solutions to \eqref{GP} is expected to asymptotically decompose into a ground state solitary wave plus small-order fluctuations.



The earlier work of Ellis, Jordan, Otto, and Turkington \cite{EJOT} established large deviation principles for mixed ensembles with bounded interaction potentials  and non-singular conditioning (see Subsection \ref{SUBSEC:COM}). In contrast, the present work extends the framework to include unbounded quartic interactions and a singular conditioning event involving an infinite renormalization constant, as typically encountered in constructive quantum field theory. Notably, the Gaussian measure with covariance $(-\dx^2+|x|^2)^{-1}$ on $\R$ is not supported on $L^2(\R)$, but rather on $\mathcal{H}^{-\eta}(\R)$ for any  $\eta>0$, or on $L^p(\R)$ for $p >2$. This implies $M(\phi)=\infty$ almost surely under the Gaussian measure, so a diverging renormalization constant (Wick renormalization) must be included in the conditional event $\{M(\phi) = D\}$. A key feature of Theorem \ref{THM:1} is that this renormalization effect vanishes at the level of the large deviation rate function.  
Since the equilibrium macrostates of the system are determined by the minimizers of the rate function $J^D$, the most probable canonical macrostates remain unaffected by the presence of renormalization. For further explanation of the main results, see Theorems~\ref{THM:1} and \ref{THM:2}.


\begin{remark}\rm 
Without the harmonic potential $|x|^2$, the infinite volume Gibbs measure becomes trivial: the measure concentrates entirely at the zero configuration, yielding $\dl_0$ (see Subsection \ref{SUBSEC:COM} for details). This is in sharp contrast to the situation with the harmonic potential $|x|^2$, where the infinite volume Gibbs measure is already nontrivial and one can meaningfully investigate its concentration around the family of minimizers, that is, the soliton manifold, in the low-temperature regime.
\end{remark}




\subsection{Main result}
In this subsection, we present the main results. Before stating the main theorem, we briefly review the construction of the mixed ensemble \eqref{Gibbs1} and introduce the relevant notation.

First, we give a precise definition of the mixed ensemble as a conditional probability measure (see equation \eqref{Gibbs4}). To this end, we first study the grand canonical Gibbs ensemble 
\begin{align}
d\rho_{\eps,A}(\phi)=Z_{\eps,A}^{-1}\exp\Big\{-\frac 1\eps H^G(\phi) \Big\} \prod_{x \in \R} d\phi(x).
\label{Gibbs10}
\end{align}

\noi
Here, $H^G$ denotes the grand canonical Hamiltonian
\begin{align}
H^G(\phi)=\frac{1}{2}\int_{\R} |\dx \phi|^2 dx+\frac 12 \int_{\R} |x|^2 |\phi|^2 dx-\frac \ld4 \int_{\R} |\phi|^4 dx+A\bigg(\int_{\R} |\phi|^2 dx \bigg)^3,
\label{Ha1}
\end{align}

\noi
where $A$ is called the chemical potential. Compared to the case $A=0$, that is, the Hamiltonian $H(\phi)$ in \eqref{Ha0}, the grand canonical Hamiltonian $H^G$ in \eqref{Ha1} with large chemical potential $A$ is sufficiently coercive for a construction of the grand canonical Gibbs ensemble \eqref{Gibbs10}.  More specifically, Lemma \ref{LEM:coercive} shows that there exists $A_0 \gg 1$ such that for any $A \ge A_0$, $H^G(\phi) >0$ unless $\phi=0$.

The grand canonical ensemble \eqref{Gibbs10} is constructed by viewing the Gibbs measure as a weighted Gaussian measure
\begin{align*}
d\rho_{\eps,A}(\phi)=Z_{\eps, A}^{-1} \exp\Big\{ -\frac 1\eps \textbf{V}(\phi)  \Big\}\mu_\eps(d\phi)
\end{align*} 

\noi
where $\mu_\eps$  is the Gaussian measure with covariance operator $\mathcal{L}=(-\dx^2+|x|^2)^{-1}$, and 
\begin{align*}
\textbf{V}(\phi)=-\frac{\ld}{4}\int_{\R} |\phi|^4 dx+A\bigg(\int_{\R} |\phi|^2 dx\bigg)^3.
\end{align*}

\noi 
However, as pointed above, the Gaussian measure $\mu_\eps$ is not supported on $L^2(\R)$, but rather on $\mathcal{H}^{-\eta}(\R)$ for any  $\eta>0$, or on $L^p(\R)$ for $p >2$, that is, $\supp \mu_\eps \subset \mathcal{H}^{-\eta}(\R) \cap L^p(\R)$ (see Subsection \ref{SUBSEC:Gauss}). This implies 
\begin{align*}
M(\phi)=\int_{\R} |\phi|^2 dx=\infty, 
\end{align*}

\noi 
almost surely under the Gaussian measure $\mu_\eps$. Therefore, to properly construct the grand canonical Gibbs ensemble, we need to apply Wick renormalization to the taming part, as follows
\begin{align}
d\rho_{\eps,A}(\phi)=Z_{\eps, A}^{-1} \exp\Big\{ -\frac 1\eps V(\phi)  \Big\}\mu_\eps(d\phi),
\label{Gibbs5}
\end{align}

\noi
where 
\begin{align}
V(\phi)=-\frac{\ld}{4} \int_{\R} |\phi|^4 dx +A\bigg( \int_{\R} : \! |\phi|^2 \! : dx  \bigg)^3.
\label{potent}
\end{align}

\noi
See \eqref{Wick} for a precise definition of $: \! |\phi|^2 \! : $. For the construction of the grand canonical ensemble with the optimal power $\g=6$ 
and sufficiently large chemical potential $A\ge A_0 \gg 1$ in the taming term $A\|\phi \|_{L^2(\R)}^\g$, see  Remark~\ref{REM:Ggibb}.

In principle, we want to define the mixed ensemble \eqref{Gibbs1} as a regular conditional distribution
\begin{align}
\rho^D_\eps(B):= \rho_{\eps,A}\big(\{\phi \in B | M^w(\phi)=D  \}\big)=\frac{\rho_{\eps,A} \big( B \cap \{ M^w(\phi)=D  \} \big) }{ \rho_{\eps,A} \big(\{ M^w(\phi)=D  \} \big)      },
\label{Gibbs20}
\end{align}

\noi
given the microcanonical constraint $M^w(\phi)=D$, where
\begin{align}
M^w(\phi)=\int_{\R} : \! |\phi|^2 \!: dx.
\label{Wick2}
\end{align}

\noi 
Then, the mixed Gibbs ensemble \eqref{Gibbs20} coincides with the formal expression \eqref{Gibbs1}, provided that the conditioning $\{M(\phi)=D\}$ in \eqref{Gibbs1} is interpreted as the Wick renormalized $L^2$ mass $M^w(\phi)$. However, to avoid technical issues related to regular conditional distributions, arising from
$\rho_{\eps,A}  \big( \{ M^w(\phi)=D  \} \big)=0$, we instead consider the following conditional measure
\begin{align}
\rho_{\eps,r}^D(B):=\rho_{\eps,A}\big\{B \big|  M^w(\phi)\in [D-r,D+r] \big\}=\frac{ \rho_{\eps,A} \big( B\cap \{M^w(\phi)\in [D-r,D+r]  \}    \big)  }{\rho_{\eps,A}  \big( \{ M^w(\phi) \in [D-r,D+r]  \} \big)  },
\label{Gibbs4}
\end{align}

\noi
where $r$ is a positive parameter that defines the thickened shell $[D-r,D+r]$. 
Then, for suitable values of $D$, and all $r>0$, $\rho_{\eps,A}  \big( \{ M^w(\phi) \in [D-r,D+r]  \} \big)>0$ (see Remark \ref{REM:wcond}). Therefore, the conditional probability $\rho_{\eps,r}^D$ is well-defined.

Compared to the definition \eqref{Gibbs20}, which is independent of the choice of the chemical potential $A$, the conditional measure $\rho^D_{\eps,r}$ depends on $A$ 
due to the presence of the small gap $r$ in the definition \eqref{Gibbs4}. However, the large deviation result remains independent of the choice of $A$. In particular, the rate function does not depend on $A$. See Theorem \ref{THM:1}.

\noi
We now define the rate function 
\begin{equation}\label{LDPrate}
\begin{split}
J^D(\phi)=
\begin{cases}
H(\phi)-\inf\limits_{M(\phi)=D} H(\phi) \quad &\text{if} \quad  \phi\in \mathcal{H}^1(\R) \quad \text{and} \quad M(\phi)=D \\
\infty \quad &\text{otherwise}.
\end{cases}
\end{split}
\end{equation}

\noi
In the following, the coupling constant $\ld>0$ in \eqref{Ha0} plays no essential role, and we may therefore assume $\ld=1$. We are now ready to state the main theorem.

\begin{theorem}\label{THM:1}
Let $\mathcal{S}=H^{-\eta}(\R)$ for any $\eta>0$, or $L^p(\R)$ for any finite $p>2$. Then, there exists $D^*>0$ such that for every $D \ge D^*$, the mixed ensemble $\rho^D_{\eps,r}$ in \eqref{Gibbs4} satisfies a large deviation principle on $\mathcal{S}$ with rate function $J^D$ and speed $\eps>0$. In other words,
\begin{itemize}
\item[(1)] For every closed set $\mathcal{C} \subset \mathcal{S}$, we have 
\begin{align*}
\limsup_{r \to 0}\limsup_{\eps \to 0} \eps \log \rho_{\eps,r}^D(\mathcal{C})&\le -\inf_{\phi \in \mathcal{C}} J^D(\phi).
\end{align*}

\medskip 

\item[(2)] For every open set $\mathcal{O}\subset \mathcal{S}$, we have 
\begin{align*}
\liminf_{r\to 0}\liminf_{\eps \to 0} \eps \log \rho_{\eps,r}^D(\mathcal{O})&\ge -\inf_{\phi \in \mathcal{O} } J^D(\phi).
\end{align*} 

\end{itemize}
 
\end{theorem}

Our result establishes a large deviation principle for a conditional probability measure \eqref{Gibbs4},  under a singular conditioning.  A striking feature is that the diverging renormalization constants, which appear in both the taming term \eqref{potent} and the conditional event \eqref{Gibbs4}, vanish in the large deviation rate function \eqref{LDPrate}. This highlights the fact that, despite the singular nature of the problem at the level of the measure,  the macroscopic behavior, as captured by the large deviation rate function $J^D$, is governed by a constrained minimization problem $J^D(\phi)=0$ under $M(\phi)=D$. This variational problem is free from renormalization effects. For a detailed explanation and a comparison with earlier literature on the large deviation principle for the mixed ensemble, see Subsections  \ref{SUBSEC:COM} and \ref{SUBSEC:proof}.


As a result of the large deviation principle in \eqref{LDP0}, the rate function 
$J^D$ assigns higher cost to configurations away from the energy minimizer, thus favoring the lowest energy level.
\begin{theorem}\label{THM:2}
Let $2<p<\infty$. For any $\dl>0$, there exists $c(\dl)>0$ such that 
\begin{align}
\rho_{\eps,r}^D\bigg( \Big\{ \inf_{Q\in \M^D} \|  \phi-Q    \|_{L^p(\R) } \ge \dl   \Big\}  \bigg) \les e^{-\frac{c(\dl)}{\eps}}
\label{CE0}
\end{align}

\noi
as $\eps \to 0$ and $r \to 0$, provided $D \ge D^*$, where $D^*$ is as in Theorem~\ref{THM:1}. Here, $\M^D$ is the soliton manifold, that is, the family of minimizers of the Hamiltonian \eqref{Ha0} under the constraint $M(\phi)=D$, 
\begin{align*}
\inf_{M(\phi)=D}H(\phi)=H(Q)
\end{align*}

\noi
for any $Q \in \M^D$.
\end{theorem}

This concentration shows that ground state solitary waves are the dominant macroscopic states, with all other configurations becoming exponentially unlikely in the low-temperature limit. As explained earlier, the Gibbs measure $\rho^D_{\eps,r}$ in \eqref{Gibbs4} is invariant under the flow of the equation \eqref{GP}, that is, $\Law(\psi(t))=\rho^D_{\eps,r}$ for every $t \in \R$. Consequently, high-probability events under $\rho^D_{\eps,r}$  reflect the long-time behavior of the dynamics governed by \eqref{GP}. In particular, the concentration estimate \eqref{CE0} implies that at low temperature, the solutions to \eqref{GP} are well approximated by
\begin{align*}
\psi(t) \approx Q+\text{small fluctuations},
\end{align*}

\noi
where $Q\in \M^D$. That is, the solution asymptotically decomposes into a ground state solitary wave plus small fluctuations.

\begin{remark}\rm 
The large mass condition $D \ge D^*$ in Theorems~\ref{THM:1} and \ref{THM:2} arises from ensuring the negativity of the minimal energy, $\inf_{M(\phi)=D}H(\phi)<0$, as discussed in Remark~\ref{REM:negmin}. Note that the argument in Remark~\ref{REM:negmin} shows that the minimal energy is always negative when $d\ge 3$, even without the large mass condition. However, to the best of the authors' knowledge, it remains unknown whether the minimal energy is negative for all masses $D>0$ in the lower-dimensional cases $d \le 2$. This negative minimal energy condition is used only in the proof of Proposition~\ref{PROP:WL2}, specifically in establishing \eqref{C000}.

\end{remark}


\subsection{Motivation and comments on the literature}\label{SUBSEC:COM}

\subsubsection{Large deviation principle for the mixed ensemble}

In \cite{EJOT}, Ellis, Jordan, Otto, and Turkington studied the mixed Gibbs ensemble associated with the equation 
\begin{align*}
i \dt \psi +L \psi +f(|\psi|^2) \psi=0
\end{align*}

\noi
on a bounded domain  $D \subset \R^d$, subject to appropriate boundary conditions. Here, $L$ is a linear operator whose negative spectrum $-L$ consists of positive eigenvalues $\{\ld_k\}_{k=1}^\infty$ satisfying $\sum_{k=1}^\infty$ $\frac{1}{\ld_k}<\infty$. A basic example is $L=\dx^2$ on a finite interval $[a,b]$. 
The class of nonlinearities $f$ considered satisfies $f(0)=0$, $\sup_{[0,\infty )} |f(a)|<\infty $, and $f'(a)>0$ (focusing condition). A typical example is $f(|\psi|^2)=\frac{|\psi|^2}{1+|\psi|^2}$. Note that this choice of linear operator $L$ excludes our case $\mathcal{L}=-\dx^2 +|x|^2$, since   $\mathcal{L}h_n=(1+2n)h_n$, where $h_n$ forms an orthonormal basis of $L^2(\R)$. Here, the eigenvalues $1+2n$ yield $\sum \frac{1}{1+2n}=\infty$, so the summability condition $\sum \frac{1}{\ld_k}$ fails. See \eqref{LZ2}.

In particular,  compared to the earlier work of Ellis, Jordan, Otto, and Turkington \cite{EJOT} on large deviation principles for mixed ensembles, where only bounded focusing interactions and non-singular conditioning $(M(\phi)<\infty)$ were considered, the present work extends the analysis to the setting of an unbounded quartic interaction and a conditional event involving an infinite renormalization constant.  


\subsubsection{Gross-Pitaevskii Gibbs measure and dynamical problem}

In \cite{BTT}, Burq, Thomann, and Tzvetkov studied the construction of the Gibbs measure \eqref{Gibbs1} on $\R$ for the focusing case $(\ld>0)$ with quartic interaction.  They also proved the invariance of the Gibbs measure under the deterministic flow of the dynamics \eqref{GP}. For the construction of Gibbs measures with higher-order focusing interaction on $\R$, that is,  $\frac{\ld}{p}\int_{\R} |\phi|^p dx$ for $p \ge 6$, in \cite{RSTW} Robert, Tolomeo, Wang, and the first author proved that the Gibbs measure cannot be constructed. Their result also establishes that the construction is only possible when $p<6$.  As for the focusing interaction on $\R^2$, see the results in \cite{Deng, RSTW} by Deng and by Robert, Tolomeo, Wang, and the first author, where the Gibbs measure can be constructed only for $p<3$, while non-construction is established for $p \ge 4$. See also \cite{DR23, DRTW} by Dinh–Rougerie and Dinh–Rougerie–Tolomeo–Wang for the study of focusing Gibbs measures associated with general anharmonic potentials $-\dx^2+|x|^s$, where $s>1$.


Regarding the defocusing case $(\ld<0)$, in \cite{BTT} Burq, Thomann, and Tzvetkov \cite{BTT} proved that the defocusing Gibbs measure is invariant under the corresponding defocusing dynamics \eqref{GP}. In a related stochastic setting, de Bouard, Debussche, and Fukuizumi \cite{DDF0, DDF1} studied the defocusing stochastic Gross–Pitaevskii equation on $\R$ and $\R^2$, formulated as the gradient flow of the Hamiltonian \eqref{Ha0} perturbed by space-time white noise. They proved the invariance of the defocusing Gibbs measures under the stochastic dynamics on both $\R$ and $\R^2$. See also the recent work of Deya-Fukuizumi-Thomann \cite{DFT25} establishing the local well-posedness of the stochastic Gross–Pitaevskii equation on $\R^3$. From the viewpoint of many-body quantum mechanics, related measures with sub-harmonic trapping potentials were studied in Lewin-Nam-Rougerie \cite{LNR18}.

Note that in the defocusing case $\ld<0$ in \eqref{Ha0}, the Hamiltonian $H(\phi)$ is coercive, that is, $H(\phi)>0$ unless $\phi=0$ (the minimizer is unique, given by $\phi=0$). As a result, there is no need to impose the conditioning event $\{ M^w(\phi)=D\}$ to construct the Gibbs measure. Therefore, the large deviation principle for the Gibbs ensemble can be established in a much simpler manner.


\subsubsection{Translation invariant Gibbs measures}
Thanks to the presence of the confining potential $|x|^2$ in the Hamiltonian $H$ in \eqref{Ha0},   the Gibbs measure \eqref{Gibbs1} is not translation invariant. As a result, the fields sampled from the Gibbs measure exhibit spatial decay at infinity. This decay allows for a direct construction of the infinite volume measure without the need for an infrared (large scale) cutoff.

In the absence of the harmonic potential $|x|^2$, the Hamiltonian becomes translation invariant, which in turn implies that the corresponding Gibbs measure is also translation invariant. In this case, a large field problem arises due to the lack of spatial decay at infinity. To address this issue,  one first constructs the finite volume measure $\rho_L$ on the torus $\T_L=\R / L\Z$ (infrared cutoff) and then takes the infinite volume limit $\rho_\infty$ as $L\to \infty$. In \cite{Rider}, Rider showed that translation invariance, combined with the strongly focusing nature of the interaction, leads to a trivial infinite volume measure $\rho_\infty=\dl_0$.  That is, the limiting measure places all of its mass on the zero path. This is in sharp contrast to the case with a harmonic potential, where the infinite volume Gibbs measure is already nontrivial.  See also the recent works \cite{TW, SSosoe, SSosoe1} for results on the infinite volume limit of focusing Gibbs measures in the translation-invariant setting.



\subsection{Structure of the proof}\label{SUBSEC:proof}
In this subsection, we present the structure of the proof.

\noi 
\textbf{Step~1~(Proposition \ref{PROP:GLDP}):} According to the definition of the conditional probability measure $\rho^D_{\eps,r}$ in \eqref{Gibbs4}, we first write
\begin{align}
\lim_{r\to 0}\lim_{\eps \to 0}\eps \log \rho^D_{\eps,r}(B)&=\lim_{r\to 0} \lim_{\eps \to 0} \eps \log \rho_{\eps,A}\big( B\cap \{M^w(\phi)\in [D-r,D+r]     \} \big) \notag \\
&\hphantom{X}-\lim_{r\to 0} \lim_{\eps \to 0} \eps \log \rho_{\eps,A}\big( \{M^w(\phi)\in [D-r,D+r]     \} \big).
\label{AG0}
\end{align}

\noi
In order to handle a general set $B$, specifically, the first term in \eqref{AG0}, we first establish a large deviation principle for the grand canonical ensemble $\rho_{\eps,A}$ \eqref{Gibbs5}
\begin{align}
\eps \log \rho_{\eps,A} (B) \approx -\inf_{\phi \in B} H^G(\phi),
\label{LPI0}
\end{align}

\noi
where $H^G$ is the grand canonical Hamiltonian in \eqref{Ha1}.



In the previous work \cite{EJOT} by Ellis, Jordan, Otto, and Turkington, the large deviation principle for the grand canonical ensemble was established by deriving it directly from the large deviation principle for the Gaussian measure $\mu_\eps$. This approach \cite[Theorem 4.4]{EJOT} was possible because the interaction potential they considered was bounded, allowing the large deviation behavior to follow automatically from that of the Gaussian measure. In contrast, our arguments in Section \ref{SEC:LDPG} address the more singular case where the potential is unbounded and involves an infinite counterterm in the taming part \eqref{potent}.


\noi 
\textbf{Step~2~(Propositions \ref{PROP:WL2} and \ref{PROP:free}):} In order to handle the conditional event in \eqref{Gibbs4}, we analyze two fundamental thermodynamic functions: (i) the microcanonical entropy
\begin{align}
\lim_{r\to 0} \lim_{\eps \to 0} \eps \log \rho_{\eps,A}\big( \{ M^w(\phi) \in [D-r,D+r]  \} \big)=-\inf_{M(\phi)=D} H^G(\phi), 
\label{SDE2}
\end{align}

\noi
and (ii) the free energy 
\begin{align*}
\lim_{\eps \to 0} \eps \log Z_{\eps,A}=-\inf_{\phi \in \mathcal{H}^1} H^G(\phi).
\end{align*}

\noi
A remarkable aspect of Proposition~\ref{PROP:GLDP} is that the diverging renormalization constants, appearing in the conditional event in \eqref{SDE2} and the taming part in \eqref{potent}, disappear at the level of the microcanonical entropy and the free energy. 


In the previous work \cite{EJOT} by Ellis, Jordan, Otto, and Turkington, the asymptotic behavior of the microcanonical entropy \cite[Proposition 4.5 (a)]{EJOT} was obtained via the large deviation principle for $\rho_{\eps,A}$ and a direct application of the contraction principle, relying on the continuity of the $L^2$-mass $M(\phi)$ on $L^2$.  In our case,  if the contraction principle were directly applicable, it would naturally yield a constraint involving the Wick-renormalized mass $M^w(\phi)$ in \eqref{SDE2}. However, such a constraint is not meaningful in the variational representation, when interpreted as an infimum constraint, that is, $\inf\limits_{M^w(\phi)=D} H^G(\phi)$. For this reason, we carry out a more careful and direct analysis of the microcanonical entropy, in which we explicitly remove the divergent renormalization constant in the low-temperature limit.


\noi 
\textbf{Step~3~(Section \ref{SEC:LDP}):} By combining \textbf{Step 1} and \textbf{Step 2}, we derive the large deviation principle for the mixed ensemble $\rho_{\eps,r}^D$ in \eqref{Gibbs4}. More precisely, for any $\phi$ satisfying $M(\phi)=D$ and arbitrary small $\dl>0$, it follows from \eqref{LPI0} and \eqref{SDE2} that as $\eps \to 0 $ and $r \to 0$,
\begin{align*}
\eps \log \rho^D_{\eps,r}(B(\phi,\dl))&= \eps \log \rho_{\eps,A}\big( B(\phi,\dl)\cap \{M^w(\phi)\in [D-r,D+r]     \} \big)\\
&\hphantom{X}-\eps \log \rho_{\eps,A}\big( \{M^w(\phi)\in [D-r,D+r]     \} \big)\\
&\approx -H^G(\phi)+\inf_{M(\phi)=D} H^G(\phi)\\
&=-H(\phi)+ \inf_{M(\phi)=D} H(\phi)=J^D(\phi),
\end{align*}

\noi
where $B(\phi, \dl)$ denotes the open ball with center $\phi$ and radius $\dl>0$ with respect to $\mathcal{S}=\mathcal{H}^{-\eta}(\R)$ or $L^p(\R)$, $p>2$.   After that, we extend the result from balls to general sets.

Once we obtain the large deviation principle (Theorem~\ref{THM:1}), Theorem~\ref{THM:2} follows as a consequence of Theorem~\ref{THM:1} together with the stability of minimizers.

\section{Notations and preliminary results}

\subsection{Notations}
When addressing regularities of functions and distributions, we  use $\eta > 0$ to denote a small constant. We usually  suppress the dependence on such $\eta > 0$ in estimates. For $a, b > 0$, $a\lesssim b$  means that
there exists $C>0$ such that $a \leq Cb$. By $a\sim b$, we mean that $a\lesssim b$ and $b \lesssim a$.

\subsection{Harmonic oscillator operator}

The operator $\mathcal{L}=-\dx^2+|x|^2$ has a positive self-adjoint extension on $L^2(\R)$ and has eigenfunctions $\{h_n\}_{n \ge 0}$ with
\begin{align} 
h_n (x) = (-1)^n c_n e^{\frac{x^2}2} \frac{d^n}{dx^n} (e^{-x^2})
\label{eigenf}
\end{align}

\noi
and $c_n = (n!)^{-\frac12} 2^{-\frac{n}2} \pi^{-\frac14}$. Then $\{h_n\}_{n \ge 0}$ is a complete normal basis of $L^2(\R)$. Let $\ld_n^2$ be the corresponding eigenvalues, that is,  $\mathcal{L} h_n = \lambda_n^2 h_n$. Then, it follows from \cite{BTT} that 
\begin{align}
\lambda_n =  \sqrt{1+2n}.
\label{LDG0}
\end{align} 

\noi
We have the following estimates on the eigenfunctions $h_n$ from \cite{YZ}
\begin{align}\label{hn}
\| h_n(x) \|_{L^p(\R)}
\lesssim \begin{cases}
\lambda_n^{-\frac13 + \frac2{3p}} & \textup{if } 2 \le p \le 4\\
\lambda_n^{-\frac16} &  \textup{if }  p \ge 4.
\end{cases},
\end{align}

\noi 
uniformly in $n\in\N$, $p\ge 2$.  We define the Sobolev spaces associated to the operator $\mathcal{L}$.
\begin{definition}
\label{DEF:sob}
For $1\le p \le \infty$
and $s \in \R$,
we define the harmonic Sobolev space $\W^{s,p} (\R)$
by the norm
\begin{align*}
\| u\|_{\W^{s,p} (\R)} = \| \mathcal{L}^{\frac{s}2} u\|_{L^p (\R)}.
\end{align*}

\noi
When $p=2$,
we write $\W^{s,2} (\R) = \H^s (\R)$
and for $u = \sum_{n=0}^\infty c_n h_n$ we have
$\|u \|_{\H^s (\R)}^2 = \sum_{n=0}^\infty \ld^{2s}_n|c_n|^2$.
\end{definition}

\noi
We recall the following Gagliardo-Nirenberg-Sobolev inequality in the harmonic Sobolev space. See \cite{RSTW}. For any finite $p>2$,
\begin{align}
\|u\|_{L^{p}(\R)}^p  \les  \| u\|^{\frac{p-2}{2}}_{\H^1 (\R)}\|u\|^{1+ \frac{p}2}_{L^2(\R)}.
\label{GNS}
\end{align}

\subsection{Gaussian measure associated with the harmonic oscillator}\label{SUBSEC:Gauss}
We define the Gaussian measure  $\mu_\eps $  whose Cameron-Martin space is $\mathcal{H}^1(\R)$, that is, covariance operator $\eps \mathcal{L}^{-1}$, formally given by
\begin{align}
d\mu_\eps =  Z_{\eps}^{-1} e^{-\frac 1{2\eps} \jb{\mathcal{L}\phi,\phi }_{L^2(\R)} }\prod_{x\in \R} d\phi(x)=Z_\eps^{-1} \prod_{n=0}^\infty e^{-\frac1{2\eps} \ld^2_n |\phi_n|^2} d \phi_n,
\label{Gaussian}
\end{align}

\noi
where $\eps>0$   denotes the temperature and  $d\phi_n$ is the Lebesgue measure on $\mathbb C$. This Gaussian measure $\mu_\eps$ is the induced probability measure under the map
\begin{align}
\label{maps}
\o \in \O \longmapsto u^\o = \sum_{n\ge 0} \frac{\sqrt{\eps} g_n (\o)}{\ld_n} h_n,
\end{align}

\noi
where $\{g_n\}_{n \in \mathbb N}$ is a sequence of independent standard complex-valued Gaussian random variables on a probability space $(\O, \F, \PP)$.
To define the Gaussian measure $\mu_\eps$ in \eqref{Gaussian} rigorously, we first introduce a finite-dimensional approximation. We begin by defining the spectral projector $\P_N$ 
\begin{align*}
\P_N u = \P_N \Big( \sum_{n=0}^\infty u_n h_n \Big) = \sum_{n=0}^N u_n  h_n,
\end{align*}

\noi
whose image is the finite dimensional space $E_N = \textup{span} \{h_0, h_1, \cdots, h_N\}$. By setting $u_N^\o=\P_N u^\o$, we have $\Law(u_N^\o)=\mu_{\eps,N}$, that is,  the pushforward of $\mu_\eps$ under $\P_N$, where 
\begin{align*}
\o \mapsto \P_Nu^\o= u_N^\o  :  = \sum_{n=0}^N \frac{g_n (\o)}{\ld_n}  h_n  .
\end{align*}

\noi
Then, given any $s >0$,  the sequence $\{ u_N^\o \}_{N \ge 1} $ is a Cauchy sequence in $L^2 (\O; \H^{- s} (\R))$ converging to $u^\o$ given in \eqref{maps}.

\noi 
It follows from \eqref{LDG0}, \eqref{maps}, and \eqref{hn} that 
\begin{align}
\E_{\mu_\eps} \Big[\| \phi\|^2_{L^2 (\R)} \Big] = \eps \sum_{n=0}^\infty \frac1{\ld_n^2}  = \infty,
\label{LZ2}
\end{align}

\noi
which implies that a typical function $\phi$ in the support of $\mu_\eps$ is not square integrable, that is, $\mu_\eps(L^2(\R))=0$. Hence, in order to define the conditional event $\{M(\phi)=D \}$ in \eqref{Gibbs1}, it is necessary to renormalize the $L^2$-norm  $\int_\R |u|^2 dx$. Given $x \in \R$, $u_N^\o(x)$  
is a mean-zero complex-valued Gaussian random variable with variance
\begin{align}
    \label{variance}
    \s_N (x) = \E \big[ |u_N^\o (x)|^2 \big] = \sum_{0 \le n \le N} \frac{h_n^2(x)}{\ld_n^2} ,
\end{align}

\noi
from which we have 
\begin{align*}
\E_{\mu_{\eps,N}} \Big[\| \phi\|^2_{L^2 (\R)} \Big]=\int_\R \s_N (x) dx = \sum_{0\le n \le N} \frac{1}{\ld_n^2}  \sim \log N \to \infty 
\end{align*}

\noi
as $N \to \infty$.
Here, $\s_N$ depends on $x\in \R$ since the Gaussian process $u^\o$ given by \eqref{maps} is not translation invariant. We can then define the Wick power $\wick{|\phi_N|^2}(x)$ via
\begin{align}
\wick{|\phi_N(x)|^2} =  |\phi_N(x)|^2 - \s_N(x).
\label{Wick}
\end{align}

\noi
It is also known, see for instance \cite[Lemma 3.6]{BTT}, that $\int_\R :|\phi_N (x)|^2: dx $ forms a Cauchy sequence in  $L^2(\H^{-s} (\R), d\mu_\eps)$ and converges to a limit, denoted by $\int_\R :|\phi (x)|^2: dx $, for any $s >0$.

\noi
On the one hand, thanks to the decay property of the eigenfunctions $h_n$ as $n\to \infty$, given in \eqref{hn},  it follows from \cite[Corollary 2.4 (i)]{RSTW} that
\begin{align*}
\E_{\mu_\eps} \Big[\| \phi \|_{L^p (\R)}^p \Big] < \infty
\end{align*}

\noi
for any finite $p>2$. Therefore, the potential energy $\frac1{p} \int_\R |u|^p dx$ in \eqref{Ha0} does not require renormalization. This implies that $\supp \mu_\eps \subset \mathcal{H}^{-s}(\R) \cap L^p(\R)$ for any $s>0$ and finite $p>2$.

\subsection{Variational characterization of the minimizers}
In this subsection, we present the family of minimizers of the Hamiltonian $H$ in \eqref{Ha0} under the constraint $M(\phi)=D$, where $D>0$. We define
\begin{align}
I(D)=\inf_{M(\phi)=D } H(\phi),
\label{Min0}
\end{align}

\noi
where $H$ is the Hamiltoninan $H$ given in \eqref{Ha0} and  $M$ is the $L^2$ mass defined in \eqref{L2mass}. Then, the minimization problem admits a family of minimizers.
\begin{lemma}\label{LEM:Min}
For every $D>0$, there exists $Q=Q_D$ in $\mathcal{H}^1(\R)$ such that $\|Q \|_{L^2(\R)}^2=D$ and
\begin{align*}
I(D)=\frac{1}{2} \int_{\R} |\dx Q|^2 dx+\frac 12 \int_{\R} |x|^2 |Q|^2 dx-\frac{\ld}{4} \int_{\R}|Q|^4 dx.
\end{align*}

\end{lemma}

For the proof of Lemma \ref{LEM:Min}, see \cite[Theorem 3.1]{Zhang}.

\begin{remark}\rm 
In the absence of the harmonic potential, the set of minimizers forms a two-dimensional manifold.  That is, if $R$ is a minimizer of \eqref{Min0} without the harmonic potential, then there exist $x_0\in\R$ and $\dr \in [0,2\pi]$ such that $R(x)=e^{i\dr} Q(x-x_0)$, where $Q$ is a fixed minimizer of \eqref{Min0}.

In the presence of the harmonic potential, it is clear that $\{e^{i \dr} Q \}_{\dr \in [0,2\pi]} \subset \mathcal{M}_D$, where $\M_D$ denotes the set of minimizers. 
Note that the Hamiltonian \eqref{Ha0} is no longer translation invariant, and thus translation is not a symmetry of the problem \eqref{Min0}. However, to the authors' knowledge, it is not known whether $\{e^{i \dr} Q \}_{\dr \in [0,2\pi]}=\mathcal{M}_D$. See \cite[Remark 1.3]{Fuk}.   
\end{remark}

\begin{remark}\rm
\label{REM:negmin}
Take a fixed Schwartz function $\phi$ that is positive  and  $\|\phi \|^2_{L^2(\R)}=1$. Define the scaling 
\begin{align*}
\phi_\zeta(x):=\sqrt{D} \zeta^{-\frac 12} \phi(\tfrac x\zeta).
\end{align*}

\noi
This ensures $ \| \phi_{\zeta} \|_{L^2(\R^d)}^2=D$ for any $\zeta>0$. Note that 
\begin{align*}
H(\phi_\eps)=\frac D2 \zeta^{-2} \int_{\R} |\nb \phi|^2 dx  +\frac {D^2}2 \zeta^2 \int_{\R} |x|^2 |\phi|^2 dx-\frac {\ld D^2}4 \zeta^{-1} \int_{\R} |\phi|^4 dx. 
\end{align*}

\noi
In the absence of a harmonic potential term, the focusing interaction (quartic term) becomes dominant as $\zeta \to \infty$. Consequently, the minimal energy of the corresponding Hamiltonian is always negative.

\noi 
In the presence of a harmonic potential term, by taking $\zeta \to 0$ and $D \to \infty $ so that the focusing interaction term becomes dominant, we obtain
\begin{align*}
\inf_{ M(\phi)=D}H(\phi) \le  H(\phi_\eps) <0.
\end{align*}

\noi 
This implies that  there exists $D^*>0$ such that for any $D>D^*$,
\begin{align}
\inf_{ M(\phi)=D}H(\phi) <0.
\label{D0}
\end{align} 

\noi
In other words, the minimal energy becomes negative for sufficiently large mass. \noi
In what follows, we use the condition that the minimal energy is negative, particularly in Proposition \ref{PROP:WL2} (see \eqref{C000}).  
\end{remark}

We now study the coercive structure of the grand canonical Hamiltonian $H^G$ in \eqref{Ha1} under a sufficiently large chemical potential $A$.
It follows from the Gagliardo-Nirenberg-Sobolev inequality \eqref{GNS} and Young's inequality that 
\begin{align}
H^G(\phi)\ge \Big( \frac 12-\dl \Big)\bigg(\int_{\R}|\dx \phi|^2 dx+\frac 12 \int_{\R} |x|^2 |\phi|^2 dx \bigg) +(A-c(\dl))\bigg(\int_{\R} |\phi|^2 dx\bigg)^3 \ge 0,
\label{coer0}
\end{align}

\noi
where $\dl>0$ is small and $c(\dl)$ is a large constant depending on $\dl>0$. 
This implies that $H^G(\phi)>0$ unless $\phi=0$, provided that the chemical potential 
$A$ is sufficiently large, that is, $A>c(\dl)$. Therefore, we obtain the following lemma.
\begin{lemma}\label{LEM:coercive}
There exists $A_0>0$ such that for any $A \ge A_0$, the grand canonical Hamiltonian $H^G$ in \eqref{Ha1} has the unique minimizer $\phi=0$.
\end{lemma}

\begin{remark}\rm \label{REM:Ggibb}
The construction of the grand canonical Gibbs measure \eqref{Gibbs5} and the role of
the Gagliardo-Nirenberg-Sobolev inequality \eqref{GNS} can be understood heuristically in terms of the associated functional integral, ignoring the renormalization
\begin{align}
Z_{\eps,A}=\int e^{\frac \ld{4\eps} \int_{\R} |\phi|^4 dx } e^{-\frac A\eps\big(\int_{\R} |\phi|^2  dx\big)^3}   e^{-\frac 1{2\eps} \jb{\mathcal{L}\phi,\phi  } }  \prod_{x\in \R} d\phi(x).
\label{funinte}
\end{align}

\noi
By applying the Gagliardo-Nirenberg-Sobolev inequality \eqref{GNS} and Young’s inequality, we can control the quartic interaction as follows: for any $\ld>0$,
\begin{align*}
\frac \ld4 \|\phi \|_{L^4(\R)}^4 \le \dl \|\phi \|^2_{\mathcal{H}^1(\R)}+c(\dl)\| \phi \|_{L^2(\R)}^6,
\end{align*}

\noi
where $\dl>0$ is small and $c(\dl)$ is a large constant depending on $\dl>0$.  This implies that 
\begin{align*}
Z_{\eps,A}\le \int e^{-\frac 1\eps(A-c(\dl))\big(\int_{\R} |\phi|^2  dx\big)^3  } e^{-\frac 1\eps \big(\frac 12-\dl\big)\jb{\mathcal{L}\phi,\phi  } }   \prod_{x\in \R} d\phi(x).
\end{align*}

\noi
Therefore, when the chemical potential $A$ is sufficiently large, that is, $A>c(\dl)$,
the functional integral is heuristically integrable with respect to the Lebesgue measure. Based on this idea, we can follow the proof in \cite[Section 4]{OSeoT} to construct the grand canonical Gibbs measure \eqref{Gibbs5}. In particular, the choice of $\g=6$ in the taming term $A \|\phi \|_{L^2(\R)}^\g$ with $A \ge A_0$ is optimal in view of the Gagliardo-Nirenberg-Sobolev inequality \eqref{GNS}. When $\g<6$ or $\g=6$ with $A$ sufficiently small, the taming effect in \eqref{funinte}  is insufficient to control the focusing quartic interaction. Therefore, in this case, we expect $Z_{\eps,A}=\infty$. 

\end{remark}

\subsection{Tools from stochastic analysis}

We recall a variational representation of the partition function for Gibbs measures, similarly to \cite{BG, BG1, OSeoT, TW}.

Let $X(t)$ denote a cylindrical Brownian motion in $L^2(\R)$, defined by
\begin{align*}
X(t) = \sum_{n\ge 0} B_n (t) h_n,
\end{align*}

\noi
where $\{h_n\}_{n\ge 0}$ is the sequence of eigenfunctions of $\mathcal{L}$ given in \eqref{eigenf} and $\{B_n\}_{n \ge 0}$ is a sequence of mutually independent complex-valued Brownian motions. We define a centered Gaussian process $W(t)$ by
\begin{align*}
W (t) = \mathcal{L}^{-\frac12} X(t) = \sum_{n\ge 0} \frac{ B_n (t)}{\ld_n} h_n.
\end{align*}

\noi
Note that $\textup{Law} (W(1)) = \mu_1=\mu$, where $\mu_1$ is the Gaussian measure defined in \eqref{Gaussian}.
In what follows, we set $W_N=W_N (1)= \P_{ N}  W (1)$ and thus $\textup{Law} (W_N(1)) = (\P_{ N})_*\mu$.

\noi 
Let $\mathbb{H}_a$ be the space of drifts,
which consists of progressively measurable processes belonging to 
$L^2 ([0,1]; L^2 (\R^d))$, $\mathbb P$-almost surely. We are now ready to present the variational representation of partition functions, known as the Boué-Dupuis variational formula \cite{BD, Ust}.

\begin{lemma}
\label{LEM:var}
Suppose that $F: C^\infty (\R) \to \R$ is measurable
such that $\E \big[ | F(W_N (1))|^p\big] < \infty$ and 
$\E \big[ | e^{- F(W_N (1))}|^q \big] < \infty$ for some $1 < p,q < \infty$ 
with $\frac1p + \frac1q =1$. Then, we have
\begin{align}
\label{var}
-\log \E \Big[ e^{F(W_N )} \Big] = \inf_{u \in \mathbb{H}_a} \E
\bigg[F\big(W_N + \P_N Z(u)  \big) + \frac12 \int_0^1 \| u (t) \|_{L^2_x (\R)}^2 dt\bigg],
\end{align}

\noi
where $W_N:=\P_N W(1)$ and $Z(u):=Z(u)(1)$ is defined by
\begin{align}
Z (u) (t) = \int_0^t \mathcal{L}^{-\frac12} u(\tau) d\tau.
\label{defZu}
\end{align}

\noi
Here,  the expectation $\E = \E_{\mathbb P}$ is respect to the underlying probability measure $\mathbb P$.
\end{lemma}

\section{Large deviation principles for the grand canonical ensemble}\label{SEC:LDPG}

\subsection{Laplace principle for the grand canonical ensemble}

In this section, we establish a Large deviation principle for the grand canonical ensemble in the low-temperature limit $\eps \to 0$. Since the large deviation principle is equivalent to the Laplace principle, we derive a Laplace principle for the grand canonical ensemble.

\begin{proposition}\label{PROP:GLDP}

Let $\ld>0$ and $A \ge A_0$, where $A_0$ is given in Lemma \ref{LEM:coercive}. As $\eps \to 0$, the family $\rho_{\eps,A}$ satisfies a Laplace principle with rate $\eps$ and rate function
\begin{equation}\label{rateG}
\begin{split}
J^G(\phi)=
\begin{cases}
\frac{1}{2}\int_{\R} |\dx \phi|^2 dx+\frac 12 \int_{\R} |x|^2 |\phi|^2 dx-\frac \ld4 \int_{\R} |\phi|^4 dx+A\Big(  \int_{\R} |\phi|^2 dx \Big)^3 \quad &\text{if} \quad  \phi\in \mathcal{H}^1(\R), \\
\infty \quad &\text{otherwise}.
\end{cases}
\end{split}
\end{equation}

\noi 
More precisely, for any continuous and bounded $f: \mathcal{S}'(\R) \to \R$, we have 
\begin{align}
\lim_{\eps \to 0} -\eps \log \int  e^{-\frac 1\eps f(\phi)}\rho_{\eps,A}(d\phi)=\inf_{\phi \in S'}\big\{ f(\phi)+J^G(\phi)  \big\}.
\label{L1}
\end{align}

\end{proposition}


\noi
In the following subsections, we prove Proposition \ref{PROP:GLDP} by combining Lemmas \ref{LEM:Ga0} and \ref{LEM:Ga1}.  Note that  the left-hand side of \eqref{L1} can be written as
\begin{align*}
-\eps \log \int  e^{-\frac 1\eps f(\phi)}\rho_{\eps,A}(d\phi)=-\eps \log \E_{\mu_\eps}\Big[e^{-\frac 1\eps (f(\phi)+V(\phi) ) }  \Big]+\eps \log \E_{\mu_\eps } \Big[ e^{-\frac 1\eps  V(\phi)  }    \Big],
\end{align*}

\noi
where
\begin{align}
V(\phi)=-\frac \ld4 \int_{\R} |\phi|^4 dx+A\bigg(\int_{\R} :\! |\phi|^2 \! : dx  \bigg)^3.
\label{potent1}
\end{align}

\noi 
If $\psi$ represents a Gaussian random field with $\Law(\psi)=\mu=\mu_1$ whose covariance is $\mathcal{L}^{-1}$, applying the linear transformation $\psi \mapsto \sqrt{\eps} \psi$, $\sqrt{\eps} \psi$ yields a
Gaussian random field with $\Law(\sqrt{\eps} \psi )=\mu_\eps$ whose covariance is $\eps \mathcal{L}^{-1}$. Therefore, 
\begin{align*}
-\eps \log \int  e^{-\frac 1\eps f(\phi)}\rho_{\eps,A}(d\phi)=-\eps \log \E_{\mu}\Big[e^{-\frac 1\eps (f( \sqrt{\eps} \phi)+V( \sqrt{\eps}\phi) ) }  \Big]+\eps \log \E_{\mu } \Big[ e^{-\frac 1\eps  V(\sqrt{\eps} \phi)  }    \Big].
\end{align*}

\noi 
By applying the variational representation (Lemma \ref{LEM:Gama}), we write
\begin{align*}
&-\eps \log \E_{\mu}\Big[e^{-\frac 1\eps (f( \sqrt{\eps} \phi)+V( \sqrt{\eps}\phi) ) }  \Big]\\
&=\inf_{u \in \Ha } \E\bigg[f(\eps^{\frac 12}W+ \eps^\frac 12Z(u) )+ V(\eps^{\frac 12}W+  \eps^\frac 12 Z(u) )+\frac \eps 2 \int_0^1 \| u(t) \|_{L^2(\R)}^2  dt  \bigg].
\end{align*}

\noi
By taking the change of variables $\eps^\frac 12 u \to u$, we have 
\begin{align*}
-\eps \log \int  e^{-\frac 1\eps f(\phi)}\rho_{\eps,A}(d\phi)= \inf_{u\in \Ha}\mathcal{F}^{V+f, \eps}(u)- \inf_{u\in \Ha}\mathcal{F}^{V,\eps}(u),
\end{align*}

\noi
where 
\begin{align}
\mathcal{F}^{V+f, \eps}(u)=\E_{\PP}\bigg[f(\eps^{\frac 12}W+ Z(u) )+ V(\eps^{\frac 12}W+ Z(u) )+\frac 12 \int_0^1 \| u(t) \|_{L^2(\R)}^2 dt   \bigg].
\label{L22}
\end{align}

\noi
In the following, we study the convergence problem in the low-temperature limit $\eps \to 0$
\begin{align}
\lim_{\eps \to 0} -\eps \log \int  e^{-\frac 1\eps f(\phi)}\rho_{\eps,A}(d\phi)= \lim_{\eps \to 0}  \inf_{u\in \Ha}\mathcal{F}^{V+f, \eps}(u)-\lim_{\eps \to 0} \inf_{u\in \Ha}\mathcal{F}^{V,\eps}(u).
\label{L2}
\end{align}

\noi
In taking the limit $\eps \to 0$, the main step is to pass the limit inside the infimum, removing both the Gaussian fluctuation $\eps^\frac 12 W$ and the  infinite counterterm added in the renormalization procedure.

\begin{lemma}\label{LEM:Ga0}
Let $f: \mathcal{S}'(\R) \to \R$  be a continuous and bounded functional. Then,
\begin{align*}
\lim_{\eps \to 0}  \inf_{u\in \Ha}\mathcal{F}^{V+f, \eps}(u)=\inf_{ u\in \Ha} \mathcal{F}^{V+f,0}(u),
\end{align*}

\noi
where $\mathcal{F}^{V+f, \eps}$ is defined in \eqref{L22}, and the limiting functional $\mathcal{F}^{V+f,0}$ is given by 
\begin{align}
\mathcal{F}^{V+f,0}(u):=\E_{\PP}\bigg[f( Z(u) )- \frac \ld4  \int_{\R} |Z(u)|^4 dx +A\bigg(\int_{\R}|Z(u)|^2 dx \bigg)^3 +\frac 12 \int_0^1 \| u(t) \|_{L^2(\R)}^2 dt   \bigg].
\label{L3}
\end{align}

\end{lemma}

We postpone the proof of Lemma \ref{LEM:Ga0} to the next subsection. Taking Lemma \ref{LEM:Ga0} for granted for now and combining it with the following lemma, we present the proof of Proposition \ref{PROP:GLDP} at the end of this subsection.

\begin{lemma}\label{LEM:Ga1}
Let $f: \mathcal{S}'(\R) \to \R$  be a continuous and bounded functional. Then,
\begin{align*}
\inf_{u\in \Ha} \mathcal{F}^{V+f,0}(u)=\inf_{\phi \in S'(\R)}\big\{f(\phi)+J^G(\phi) \big\},
\end{align*}

\noi
where  $\mathcal{F}^{V+f,0}$ is defined in \eqref{L3}, and $J^G$ is the rate function given in \eqref{rateG}.
\end{lemma}

\begin{proof}
We first prove the upper bound 
\begin{align*}
\inf_{u\in \Ha} \mathcal{F}^{V+f,0}(u)\le \inf_{\phi \in S'(\R)}\big\{f(\phi)+J^G(\phi) \big\}.
\end{align*}

\noi
Taking the infimum over processes of the form
\begin{align}
u(t)=\mathcal{L}^\frac 12\phi
\label{L4}
\end{align}

\noi
for every $ 0 \le t \le 1$, where $\phi \in \mathcal{H}^1(\R)$ is a deterministic function, we have 
\begin{align}
Z(u)=\int_0^1 \mathcal{L}^{-\frac 12} u(t) dt=\phi. 
\label{Zinf0}
\end{align}

\noi
With the drift chosen as in \eqref{L4}, the entropy term can be written as
\begin{align}
\frac 12 \int_{0}^1  \| u(t) \|_{L^2(\R)}^2 dt=\frac 12 \int_0^1 \| \mathcal{L}^\frac 12 \phi \|_{L^2(\R)}^2 dt=\| \phi \|_{\mathcal{H}^1(\R)}^2.
\label{entc0}
\end{align} 

\noi
Therefore, combining \eqref{L3}, \eqref{L4}, \eqref{Zinf0}, and \eqref{entc0}, we obtain
\begin{align*}
\inf_{u\in \Ha} \mathcal{F}^{V+f, 0}(u)&\le \inf_{\substack{u(t)\equiv\mathcal{L}^\frac 12\phi \\ \phi \in \mathcal{H}^1  }  } \mathcal{F}^{V+f, 0}(u)\\
&= \inf_{\substack{ u(t)\equiv\mathcal{L}^\frac 12\phi \\ \phi \in \mathcal{H}^1  }   }\bigg\{ f(\phi)-\frac \ld4 \int_{\R} |\phi|^4 dx+A\bigg( \int_{\R} |\phi|^2 dx \bigg)^3  +\frac 12 \int_{\R} |\mathcal{L}^\frac 12 \phi|^2dx \bigg\}\\
&=\inf_{\phi \in \mathcal{H}^1}\big\{f(\phi)+J^G(\phi) \big\}\\
&=\inf_{\phi \in \mathcal{S}'}\big\{f(\phi)+J^G(\phi) \big\},
\end{align*}

\noi
where the last equality follows from the density of $\mathcal{H}^1$ in $\mathcal{S}'$.

Next, we prove the lower bound
\begin{align*}
\inf_{u\in \Ha} \mathcal{F}^{V+f,0}(u)\ge \inf_{\phi \in S'(\R)}\big\{f(\phi)+J^G(\phi) \big\}.
\end{align*}

\noi
By applying \eqref{defZu}, Minkowski’s inequality, and the Cauchy–Schwarz inequality,
\begin{align}
\| Z(u)\|_{\mathcal{H}^1(\R)}^2 \le \int_0^1 \|  u(t)\|_{L^2(\R)}^2 dt.
\label{LL3}
\end{align}

\noi
It follows from \eqref{L3} and \eqref{LL3} that 
\begin{align*}
&\inf_{u\in \Ha} \mathcal{F}^{V+f,0}(u)\\
&\ge  \inf_{u\in \Ha} \E\bigg[ f(Z(u))-\frac \ld4\int_{\R}|Z(u)|^4 dx+A\bigg( \int_{\R} |Z(u)|^2 dx \bigg)^3+\frac 12 \| Z(u) \|^2_{\mathcal{H}^1(\R)}     \bigg].
\end{align*}

\noi
Note that for every $u \in \Ha$, we have  
\begin{align*}
&f(Z(u))-\frac \ld4\int_{\R}|Z(u)|^4 dx+A\bigg( \int_{\R} |Z(u)|^2 dx \bigg)^3+\frac 12 \| Z(u) \|^2_{\mathcal{H}^1(\R)}    \\
&=f(Z(u))+J^G(Z(u))\\
&\ge \inf_{\phi \in \mathcal{S}'}\big\{f(\phi)+J^G(\phi) \big\}, 
\end{align*}

\noi
which implies that 
\begin{align*}
\inf_{u\in \Ha} \mathcal{F}^{V+f,0}(u) \ge \inf_{\phi \in \mathcal{S}'}\big\{f(\phi)+J^G(\phi) \big\}.
\end{align*}

\noi
This completes the proof of Lemma \ref{LEM:Ga1}.

\end{proof}

We are now ready to prove Proposition \ref{PROP:GLDP}.

\begin{proof}[Proof of Proposition \ref{PROP:GLDP}]
It follows from \eqref{L2} and Lemmas \ref{LEM:Ga0}, \ref{LEM:Ga1} that
\begin{align*}
\lim_{\eps \to 0} -\eps \log \int  e^{-\frac 1\eps f(\phi)}\rho_{\eps,A}(d\phi)&= \lim_{\eps \to 0}  \inf_{u\in \Ha}\mathcal{F}^{V+f, \eps}(u)-\lim_{\eps \to 0} \inf_{u\in \Ha}\mathcal{F}^{V,\eps}(u)\\
&=\inf_{\phi \in S'(\R)}\big\{f(\phi)+J^G(\phi) \big\}-\inf_{\phi \in S'(\R)} J^G(\phi)\\
&=\inf_{\phi \in S'(\R)}\big\{f(\phi)+J^G(\phi) \big\}, 
\end{align*}

\noi
where in the last line we used 
\begin{align*}
\inf_{\phi \in S'(\R)} J^G(\phi)=\inf_{\phi \in S'(\R)} H^G(\phi)=0
\end{align*}

\noi
since $H^G$ has a unique minimizer $\phi=0$, as stated in Lemma \ref{LEM:coercive}.
This completes the proof of Proposition \ref{PROP:GLDP}.

\end{proof}

\subsection{Gamma convergence}
In this subsection, we present the proof of Lemma \ref{LEM:Ga0} at the end of this subsection, using the Gamma convergence approach. We first recall the definition of Gamma convergence.
\begin{definition}\label{DEF:1}
Let $X$ be a first-countable topological space, and let $F_n:X\to \cj \R$ be a sequence of functionals on $X$. Then, $F_n$ is said to  $\G$-converge to the $\G$-limit $F:X \to \cj \R$ if the following two conditions hold:
\begin{itemize}
\item[(i)] Let $x \in X$. For every sequence $\{x_n\}_{n\ge 1}\subset X$ with $x_n \to x$ as $n \to \infty$,
\begin{align*}
F(x)\le \liminf_{n\to \infty} F_n(x_n).
\end{align*}

\smallskip 

\item[(ii)] For every $x \in X$, there exists a sequence $\{x_n\}_{n \ge 1} \subset X$ converging to $x$, called a recovery sequence, such that
\begin{align*}
\limsup_{n\to \infty} F_n(x_n) \le F(x).
\end{align*}

\end{itemize}

\end{definition}

\begin{definition}
Let $F_n$ be a sequence of functionals as in Definition \ref{DEF:1}. 
The sequence $F_n:X \to \cj \R$ is said to be equicoercive if there exists a compact set  $\mathcal{K} \subset X$ such that for all $n\in \N$,
\begin{align*}
\inf_{x \in \mathcal{K} }F_n(x)=\inf_{x\in X} F_n(x).
\end{align*}

\end{definition}

Combining Gamma-convergence with equicoercivity, we obtain the convergence of the infimum.

\begin{proposition}\label{PROP:Ga}
Suppose that $F_n$ $\G$-converges to the $\G$-limit $F$, and the sequence $\{F_n\}_{n\in \ge 1}$ is equicoercive. Then, $F$ attains its minimum, and 
\begin{align*}
\min_{x \in X} F(x)=\lim_{n\to \infty} \inf_{x\in X} F_n(x).
\end{align*} 
\end{proposition}

\noi
For the proof of Proposition \ref{PROP:Ga}, see \cite{Maso}.

In the following, we prove the limit stated in Lemma \ref{LEM:Ga0}:
\begin{align*}
\lim_{\eps \to 0}  \inf_{u\in \Ha}\mathcal{F}^{V+f, \eps}(u)=\inf_{ u\in \Ha} \mathcal{F}^{V+f,0}(u)
\end{align*}

\noi
by showing that $\mathcal{F}^{V+f, \eps}$ $\G$-converges to the limiting functional $\mathcal{F}^{V+f, 0}$ and the family  $\{ \mathcal{F}^{V+f, \eps} \}_{\eps>0} $ is equicoercive.

To make use of Gamma convergence, we need to adjust the variational setting to ensure the necessary compactness. Instead of minimizing over the drift $u \in \Ha$, we relax the variational problem by minimizing over the  law of the pair $(\mathbb{W}, u)$, where the enhanced data set $\mathbb{W}=(W,\, :\!|W|^2 \!:)$ is fixed and $u$ varies over $\Ha$. This approach is based on \cite{BG}.

\begin{definition}
We define 
\begin{align*}
\mathcal{X}:=\bigg\{ \nu: \nu=\Law_{\mathbb{P} }(\mathbb{W}, u ) \; \text{for} \; u \in \Ha  \; \text{and} \; \int_0^1 \| u(t) \|_{L^2(\R)}^2 dt <\infty  \bigg\},
\end{align*}

\noi
where $\mathbb{W}=(W, \, :\! |W|^2 \!: )$ is fixed and $u$ varies over $\Ha$. Note that all elements of $\mathcal{X}$ have first marginal given by $\Law(\mathbb{W})$.
\end{definition}

To ensure compactness, we need to complete the space $\mathcal{X}$ of measures with respect to a suitable topology.
\begin{definition}
We define 
\begin{align*}
\cj{\mathcal{X}}:=\bigg\{& \nu: \, \text{there exist a sequence} \, \{\nu_n \}_{n \ge 1}\subset \mathcal{X} \; \text{such that} \; \text{$\nu_n \to \nu$ weakly}  \\
&\hphantom{XX} \text{and} \quad \sup_{n\ge 1} \E_{\nu_n}\bigg[ \int_0^1 \| u(t) \|_{L^2(\R)}^2 dt   \bigg]<\infty \bigg\}.
\end{align*}

\noi
Specifically, we endow $\cj {\mathcal{X}}$ with the following topology: a seuqnce $\{\nu_n\}_{n \ge 1}$ in $\cj {\mathcal{X}}$ is said to converge to $\nu$ if 
\begin{itemize}
\item[(1)] $\nu_n $ converges weakly to $\nu$
\smallskip
\item[(2)] $\sup_{n\ge 1} \E_{\nu_n}\bigg[ \int_0^1 \| u(t) \|_{L^2(\R)}^2 dt   \bigg]<\infty$.
\end{itemize}

\noi
As in the case of $\mathcal{X}$, all elements of $\cj {\mathcal{X}}$  have their first marginal given by $\Law(\mathbb{W})$.

\end{definition}

With a slight abuse of notation, we define for any $\nu \in \mathcal{X}$,
\begin{align}
\mathcal{F}^{V+f, \eps }(\nu)&=\E_{\nu}\bigg[f(\eps^{\frac 12}W+ Z(u) )+ V(\eps^{\frac 12}W+ Z(u) )+\frac 12 \int_0^1 \| u(t) \|_{L^2(\R)}^2 dt   \bigg] \label{L222}\\
\mathcal{F}^{V+f,0 }(\nu)&=\E_{\nu}\bigg[f( Z(u) )+ V( Z(u) )+\frac 12 \int_0^1 \| u(t) \|_{L^2(\R)}^2 dt   \bigg] \label{L33}
\end{align}

\noi 
By definition of the measure $\nu=\Law_{\PP}(\mathbb{W},u) \in \mathcal{X}$, where $\mathbb{W}$ is fixed and $u$ varies over $\Ha$, we have  
\begin{align}
\mathcal{F}^{V+f, \eps }(\nu)&= \mathcal{F}^{V+f, \eps }(u) \label{Gaa0} \\
\mathcal{F}^{V+f, 0 }(\nu)&= \mathcal{F}^{V+f, 0 }(u).
\label{Gaa1}
\end{align}

\noi 
It follows from \cite[Lemma 8]{BG1}  that
\begin{align*}
\inf_{ \nu \in \cj  { \mathcal{X} } } \mathcal{F}^{V+f, \eps }(\nu)&=\inf_{ \nu \in  \mathcal{X}  } \mathcal{F}^{V+f, \eps }(\nu)=\inf_{u \in \Ha}\mathcal{F}^{V+f, \eps }(u)  \\
\inf_{ \nu \in \cj  { \mathcal{X} } } \mathcal{F}^{V+f,0}(\nu)&=\inf_{ \nu \in  \mathcal{X}  } \mathcal{F}^{V+f,0}(\nu)=\inf_{u \in \Ha}\mathcal{F}^{V+f, 0 }(u),
\end{align*}

\noi
In other words, the variational problems over $\mathcal{X}$ and $\cj {\mathcal{X}}$ are equivalent.  In particular,  the second equality in each line follows from \eqref{Gaa0} and \eqref{Gaa1}, respectively.  This shows that the relaxed variational problem can be studied over $\cj {\mathcal{X}}$, rather than minimizing over the drift $u\in \Ha$.


In the following, we prove the Gamma convergence of the functional $\mathcal{F}^{V+f, \eps}$ to its Gamma limit $\mathcal{F}^{V+f,0}$ on $\cj {\mathcal{X}}$.
\begin{lemma}\label{LEM:Ga2}  
Let $f: \mathcal{S}'(\R) \to \R$  be a continuous and bounded functional. Then, $\mathcal{F}^{V+f, \eps}$ $\G$-converges to the limiting functional $\mathcal{F}^{V+f, 0}$ on $\cj {\mathcal{X}}$, where   $\mathcal{F}^{V+f,\eps}$ and $\mathcal{F}^{V+f,0}$ are defined in \eqref{L222} and \eqref{L33}, respectively.

\end{lemma}

\begin{proof}
We first prove the liminf inequality. Let $\nu \in \cj{\mathcal{X}}$. Then, for any sequence $\{ \nu_\eps \}_{\eps>0} \subset \cj{\mathcal{X}}$ converging to $\nu$ in $\cj{\mathcal{X}}$, we show that 
\begin{align}
\mathcal{F}^{V+f,0}(\nu) \le \liminf_{\eps \to 0}  \mathcal{F}^{V+f,\eps} (\nu_\eps ). 
\label{Ga000}
\end{align}

\noi
We may assume that 
\begin{align}
\mathcal{U}(u):=\sup_{\eps >0} \E_{ \nu_\eps } \bigg[A\bigg( \int_{\R} |Z(u)|^2 dx \bigg)^3  +\frac 12 \int_0^1  \| u(t) \|^2_{L^2(\R)} dt  \bigg]<\infty.
\label{LL4}
\end{align}

\noi
Otherwise, the result is trivial.

From the Skorokhod representation theorem of \cite{Jak}, 
there exist random variables $(\mathbb{Y}_\eps, r_\eps)_\eps$ and $(\mathbb{Y}, r)$ defined on a common probability space $(\O, \mathcal{G}, \mathbb{Q})$ such that 
\begin{align}
\Law_{\mathbb{Q}}(\mathbb{Y}_\eps, r_\eps )&=\nu_\eps \label{Law1} \\
\Law_{\mathbb{Q}}(\mathbb{Y}, r )&=\nu  \label{Law2}
\end{align}

\noi
and $\mathbb{Y}_\eps$ converges to $\mathbb{Y}$ in $\mathcal{H}^{-\eta} \times \mathcal{H}^{-2\eta}$, $\mathbb{Q}$-almost surely, and $r_\eps $ converges to $r$ in $\Ha$, $\mathbb{Q}$-almost surely. Combined with Lemma \ref{LEM:Gamma1}, this implies that
\begin{align}
&\liminf_{\eps \to 0} \mathcal{F}^{V+f,\eps }(\nu_\eps) \notag \\
&=\liminf_{\eps \to 0} \E_{\mathbb{Q}}\Bigg[ f(\eps Y^\frac 12 +Z(r_\eps) )+\mathcal{E}(\mathbb{Y}_\eps, Z(r_\eps), \eps )   -\frac \ld4 \int_{\R} |Z(r_\eps )|^4 dx \notag  \\
&\hphantom{XXXXXXXXXX} +A\bigg( \int_{\R} |Z(r_\eps)|^2 dx \bigg)^3 +\frac 12 \int_0^1 \|r_\eps(t) \|_{L^2}^2 dt \Bigg],
\label{Ga00}
\end{align}

\noi
where $Y$ is the first component of $\mathbb{Y}$.

From \eqref{LL3} and \eqref{coer0}, we have 
\begin{align}
-\frac \ld4 \int_{\R} |Z(r_\eps )|^4 dx +A\bigg( \int_{\R} |Z(r_\eps)|^2 dx \bigg)^3 +\frac 12 \int_0^1 \|r_\eps(t) \|_{L^2}^2 dt \ge H^G(Z(r_\eps)) \ge 0
\label{Ga0}
\end{align}

\noi
It follows from \eqref{Ga0} and  Fatou's lemma  that 
\begin{align}
&\liminf_{\eps \to 0} \E_{\mathbb{Q} } \Bigg[-\frac \ld4 \int_{\R} |Z(r_\eps )|^4 dx +A\bigg( \int_{\R} |Z(r_\eps)|^2 dx \bigg)^3 +\frac 12 \int_0^1 \|r_\eps(t) \|_{L^2(\R)}^2 dt   \Bigg] \notag \\
&\ge \E_{\mathbb{Q} } \Bigg[-\frac \ld4 \int_{\R} |Z(r )|^4 dx +A\bigg( \int_{\R} |Z(r)|^2 dx \bigg)^3 +\frac 12 \int_0^1 \|r(t) \|_{L^2(\R)}^2 dt   \Bigg].
\label{Ga1} 
\end{align}

\noi
From \eqref{Law1}, Lemma \ref{LEM:error}, and \eqref{LL4}, we have 
\begin{align}
\E_{\mathbb{Q}} \big[|\mathcal{E}(\mathbb{Y}_\eps, Z(r_\eps), \eps )  |  \big]= \E_{\nu_\eps} \big[|\mathcal{E}(\mathbb{W}, Z(u), \eps ) |   \big]  \les \eps^\frac 12 \cdot C+\eps^\frac 12  \mathcal{U}(u)=O(\eps^\frac 12),
\label{Ga2}
\end{align}

\noi
where $C$ arises from computing the expected values of the higher moments for each component of $\mathbb{W}=(W, \, :\! |W|^2 \!: )$ in $\mathcal{H}^{-\eta}$.

Combining \eqref{Ga00}, \eqref{Ga1}, \eqref{Law2}, and taking $\liminf$s on both sides of \eqref{Ga2} yields 
\begin{align*}
&\liminf_{\eps \to 0} \mathcal{F}^{V+f,\eps }(\nu_\eps) \\
&\ge \E_{\mathbb{Q} } \Bigg[f( Z(r)) -\frac \ld4 \int_{\R} |Z(r )|^4 dx +A\bigg( \int_{\R} |Z(r)|^2 dx \bigg)^3 +\frac 12 \int_0^1 \|r(t) \|_{L^2(\R)}^2 dt   \Bigg]\\
&= \mathcal{F}^{V+f,0}(\nu).
\end{align*}

\noi
This completes the proof of \eqref{Ga000}.

We now prove the limsup inequality. Let $\nu \in \cj  { \mathcal{X} }$.  We choose the recovery sequence  $\{\nu_\eps\}$ by setting  $\nu_\eps =\nu$ for every $\eps >0$.
Then, clearly, $\nu_\eps $ converges weakly to $\nu$. In the following, we show that for this recovery sequence $\nu_\eps=\nu$, 
\begin{align}
\limsup_{\eps \to 0}  \mathcal{F}^{V+f,\eps} (\nu^\eps ) \le \mathcal{F}^{V+f,0}(\nu). 
\label{Ga30}
\end{align}

\noi
We may assume that $\mathcal{F}^{V+f,0}(\nu)<\infty$. Otherwise, the statement is trivial. Note that 
\begin{align}
&\limsup_{\eps \to 0}  \mathcal{F}^{V+f,\eps} (\nu^\eps )=\limsup_{\eps \to 0}  \mathcal{F}^{V+f,\eps} (\nu ) \notag \\
&=\limsup_{\eps \to 0} \E_{\nu }\Bigg[ f(\eps^\frac 12 W+Z(u)) +\mathcal{E}(\mathbb{W}, Z(u),\eps )  -\frac \ld4\int_{\R} |Z(u)|^4 dx \notag \\
&\hphantom{XXXXXXXXXXXXXX}+A\bigg(\int_{\R}|Z(u)|^2 dx \bigg)^3 +\frac 12 \int_0^1 \| u(t) \|_{L^2(\R)}^2 dt\Bigg]
\label{Ga3}
\end{align}

\noi
From Lemma \ref{LEM:error}, we have 
\begin{align}
\E_{\nu } \big[|\mathcal{E}(\mathbb{W}, Z(u), \eps ) |   \big]  \les \eps^\frac 12 \cdot C+\eps^\frac 12  \mathcal{F}^{f+V,0 }(\nu) =O(\eps^\frac 12),
\label{Ga4}
\end{align}

\noi
where $C$ arises from computing the expected values of the higher moments for each component of $\mathbb{W}=(W, \, :\! |W|^2 \!: )$ in $\mathcal{H}^{-\eta}$. Hence, it follows from \eqref{Ga3} and \eqref{Ga4} that
\begin{align*}
&\limsup_{\eps \to 0}  \mathcal{F}^{V+f,\eps} (\nu^\eps )\\
&\le  \E_{\nu} \Bigg[ f(Z(u)) -\frac \ld4\int_{\R} |Z(u)|^4 dx+A\bigg(\int_{\R}|Z(u)|^2 dx \bigg)^3 +\frac 12 \int_0^1 \| u(t) \|_{L^2(\R)}^2 dt  \Bigg]\\
&=\mathcal{F}^{f+V,0}(\nu).
\end{align*}

\noi
This completes the proof of \eqref{Ga30}.

\end{proof}

In the following lemma, we establish equicoercivity.
\begin{lemma}\label{LEM:Ga3}
The family  $\{ \mathcal{F}^{V+f, \eps} \}_{0<\eps\le 1} $, defined in \eqref{L22}, is equicoercive on $\cj {\mathcal{X}  }$.
\end{lemma}

\begin{proof}
We show that there exists a compact set $\mathcal{K} \subset \cj{\mathcal{X}}$ such that  
\begin{align*}
\inf_{\nu \in \cj{\mathcal{X}} } \mathcal{F}^{f+V, \eps}(\nu)=\inf_{\nu \in \mathcal{K}}\mathcal{F}^{f+V,\eps}(\nu).
\end{align*}

\noi
for every $0<\eps \le 1$.  Given $M>0$, to be chosen later, we set
\begin{align*}
\mathcal{K}:=\Bigg\{  \nu \in \cj{ \mathcal{X}  } : \E_{\nu }\Big[  \| Z(u) \|_{L^2}^6  \Big] +\E_{\nu }  \bigg[ \int_0^1 \| u(t) \|_{L^2(\R)}^2 dt \bigg]   \le M   \Bigg\}.
\end{align*}

\noi
Then, it follows from \cite[Lemma 10]{BG} that $\mathcal{K}$ is a compact set.

Note that from Lemma \ref{LEM:error}, we have 
\begin{align}
\mathcal{F}^{V+f, \eps}(\nu) \ge -C +(1-\dl) \E_{\nu}\Bigg[ A\bigg( \int_{\R} |Z(u)|^2 dx \bigg)^3+ \frac 12 \int_{0}^1 \| u(t) \|_{L^2(\R)}^2 dt  \Bigg]
\label{Ga5}
\end{align} 

\noi
for some small $\dl>0$, where $C$ comes from the expected values of the higher moments for each component of $\mathbb{W}=(W, \, :\! |W|^2 \!: )$ in $\mathcal{H}^{-\eta}$. 
In particular, Lemma \ref{LEM:error} also implies 
\begin{align}
\sup_{0<\eps \le 1} \inf_{\nu \in \cj{\mathcal{X}} } \mathcal{F}^{f+V, \eps}(\nu)  <\infty. 
\label{Ga6}
\end{align}

Combining \eqref{Ga5} and \eqref{Ga6} with the fact that  $M$ is taken sufficiently large, we obtain 
\begin{align*}
\inf_{\nu \notin \mathcal{K}} \mathcal{F}^{V+f, \eps}(\nu) \ge c_1M-C >  \sup_{0<\eps\le 1} \inf_{\nu \in \cj{\mathcal{X}} } \mathcal{F}^{f+V, \eps}(\nu)   
\end{align*}

\noi
for some $c_1>0$. This yields 
\begin{align*}
\inf_{\nu \in \cj{\mathcal{X}} } \mathcal{F}^{f+V, \eps}(\nu)=\inf_{\nu \in \mathcal{K}}\mathcal{F}^{f+V,\eps}(\nu)
\end{align*}

\noi
for every $0<\eps \le 1$, as desired.
\end{proof}

We are now ready to present the proof of Lemma \ref{LEM:Ga0}.
\begin{proof}[Proof of Lemma \ref{LEM:Ga0}]
Combining Lemmas \ref{LEM:Ga2} and \ref{LEM:Ga3}, based on Proposition \ref{PROP:Ga}, yields
\begin{align}
\lim_{\eps \to 0}  \inf_{\nu \in \cj {\mathcal{X}}}\mathcal{F}^{V+f, \eps}(\nu)=\inf_{ \nu \in \cj  { \mathcal{X} } } \mathcal{F}^{V+f,0}(\nu).
\label{Ga80}
\end{align}

It follows from \cite[Lemma 8]{BG1}  that 
\begin{align}
\inf_{ \nu \in \cj  { \mathcal{X} } } \mathcal{F}^{V+f, \eps }(\nu)&=\inf_{ \nu \in  \mathcal{X}  } \mathcal{F}^{V+f, \eps }(\nu)\label{Ga7}\\
\inf_{ \nu \in \cj  { \mathcal{X} } } \mathcal{F}^{V+f,0}(\nu)&=\inf_{ \nu \in  \mathcal{X}  } \mathcal{F}^{V+f,0}(\nu).
\label{Ga8}
\end{align}

\noi
By the definition of the measure $\nu=\Law_{\PP}(\mathbb{W},u) \in \mathcal{X}$, where $\mathbb{W}$ is fixed and $u$ varies over $\Ha$, we have  
\begin{align}
\inf_{ \nu \in  \mathcal{X}  } \mathcal{F}^{V+f, \eps }(\nu)&=\inf_{u \in \Ha } \mathcal{F}^{V+f, \eps }(u) \label{Ga9}\\
\inf_{ \nu \in  \mathcal{X}  } \mathcal{F}^{V+f, 0 }(\nu)&=\inf_{u \in \Ha } \mathcal{F}^{V+f, 0 }(u).
\label{Ga10}
\end{align}

\noi
By using \eqref{Ga9}, \eqref{Ga7}, \eqref{Ga80}, \eqref{Ga8}, and \eqref{Ga10}, we obtain
\begin{align*}
\lim_{\eps \to 0}  \inf_{u \in \Ha } \mathcal{F}^{V+f, \eps }(u)= \lim_{\eps \to 0}  \inf_{\nu \in \cj {\mathcal{X}}}\mathcal{F}^{V+f, \eps}(\nu)=\inf_{ \nu \in \cj  { \mathcal{X} } } \mathcal{F}^{V+f,0}(\nu)=\inf_{u \in \Ha } \mathcal{F}^{V+f, 0 }(u).
\end{align*}

\noi
This completes the proof of Lemma \ref{LEM:Ga0}.

\end{proof}

\section{Asymptotic analysis of thermodynamic functions}

In this section, we analyze two fundamental thermodynamic functions: (i) the microcanonical entropy $\rho_{\epsilon,A}(\{M^w(\phi) \in [D - r, D + r]\})$ and (ii) the free energy $\epsilon \log Z_{\epsilon,A}$.

Recall that the mixed Gibbs ensemble $\rho^D_{\eps,r}$ \eqref{Gibbs4} is defined as the conditional probability distribution  
\begin{align*}
\rho_{\eps,r}^D(B)=\rho_{\eps,A}\big\{B \big|  M^w(\phi)\in [D-r,D+r] \big\}=\frac{ \rho_{\eps,A} \big( B\cap \{M^w(\phi)\in [D-r,D+r]     \} \big)  }{\rho_{\eps,A}  \big( \{ M^w(\phi) \in [D-r,D+r]  \} \big)  },
\end{align*}

\noi
where $B$ is a measurable set. In order to study the conditional distribution, we investigate the asymptotic behavior of the microcanonical entropy in the low-temperature limit.

\begin{proposition}\label{PROP:WL2}
Let $D^*>0$ be as in \eqref{D0}. Then, for any $D>D^*$, we have 
\begin{align*}
\lim_{r \to 0}\lim_{\eps \to 0}\eps \log \rho_{\eps, A} \big( \{ M^w(\phi) \in [D-r,D+r] \} \big)=-\inf\limits_{\substack{ M(\phi)=D}} H^G(\phi),
\end{align*}

\noi
where $M^w$ is the Wick renormalized $L^2$ norm in \eqref{Wick2} and $H^G$ is the grand canonical Hamiltonian  in \eqref{Ha1}.

\end{proposition}


\begin{proof}

We first prove the upper bound 
\begin{align*}
\limsup_{r \to 0}\limsup_{\eps \to 0}\eps \log \rho_{\eps, A} \big( \{ M^w(\phi) \in [D-r,D+r] \} \big) \le -\inf\limits_{\substack{ M(\phi)=D}} H^G(\phi).
\end{align*}

\noi 
From the definition $\rho_{\eps,A}$ of the grand canonical ensemble \eqref{Gibbs5}, we have
\begin{align}
&\eps \log \rho_{\eps, A} \big( \{ M^w(\phi) \in [D-r,D+r] \} \big) \notag \\
&=\eps \log Z_{\eps,A}(\big( \{ M^w(\phi) \in [D-r,D+r] \} \big))  -\eps\log Z_{\eps,A},
\label{WL1}
\end{align}

\noi
where $Z_{\eps,A}$ is the partition function and 
\begin{align}
Z_{\eps,A}\big( \{ M^w(\phi) \in [D-r,D+r] \} \big)&=\int_{\{ M^w(\phi) \in [D-r,D+r] \}}  e^{-\frac 1\eps V(\phi) } \mu_{\eps}(d\phi) \notag \\
&\le \int \exp\Big\{ -\frac 1\eps V(\phi) \ind_{ \{ M^w(\phi) \in [D-r,D+r] \}   }  \Big\} \mu_\eps(d\phi).
\label{WI0}
\end{align}

\noi 
From Proposition \ref{PROP:free}, the free energy is determined by the minimal energy configuration as follows
\begin{align}
\lim_{\eps \to 0}\eps \log Z_{\eps, A}=-\inf_{\phi \in \mathcal{H}^{1} }H^G(\phi)=0,
\label{F00}
\end{align}

\noi
where we used the fact that $\phi=0$ is the unique minimizer for the Hamiltonian $H^G$. See Lemma \ref{LEM:coercive}. Therefore, it suffices to consider the first term in \eqref{WL1}.

Note that if $u$ represents a Gaussian random variable with $\Law(u)=\mu_1$, applying the linear transformation $u \mapsto \sqrt{\eps} u$, $\sqrt{\eps} u$ yields a
Gaussian random variable with $\Law(\sqrt{\eps}u )=\mu_\eps$. Therefore, 
\begin{align*}
&\E_{\mu_\eps}\bigg[ \exp\Big\{ -\frac 1\eps V(\phi) \ind_{ \{ M^w(\phi) \in [D-r,D+r] \}   }  \Big\} \bigg]\\
&=\E_{\mu}\bigg[ \exp\Big\{ -\frac 1\eps V(\sqrt{\eps}\phi) \ind_{ \{ M^w(\sqrt{\eps} \phi) \in [D-r,D+r] \}   }  \Big\} \bigg].
\end{align*}

\noi
From the variational representation of the Gibbs measure (Lemma \ref{LEM:Gama}), combined with the indicator function, we obtain
\begin{align*}
& \eps \log \E_{\mu}\bigg[ \exp\Big\{ -\frac 1\eps V(\sqrt{\eps}\phi) \ind_{ \{ M^w(\sqrt{\eps} \phi) \in [D-r,D+r] \}   }  \Big\} \bigg]\\
&=\sup_{u\in \Ha} \E\Bigg[ -V(\eps^\frac 12 W+\eps^\frac 12 Z(u) )\ind_{  \{ M^w(\eps^{\frac 12}W+ \eps^\frac 12 Z(u) ) \in [D-r,D+r]  \}   } -\frac \eps{2} \int_0^1 \| u(t)\|_{L^2(\R)}^2 dt   \Bigg]
\end{align*}

\noi
where
\begin{align*}
V(\eps^{\frac 12}W+\eps^\frac 12 Z(u))=-\frac \ld4\bigg(\int_{\R}|\eps^{\frac 12}W+\eps^\frac 12 Z(u)|^4 dx \bigg)+A\bigg|\int_{\R} :\! |\eps^\frac 12 W+\eps^\frac 12 Z(u)|^2 \!: dx  \bigg|^3.
\end{align*}

\noi
By applying the change of variables $\eps^\frac 12 u \to u$, we have 
\begin{align}
&\eps \log \E_{\mu}\bigg[ \exp\Big\{ -\frac 1\eps V(\sqrt{\eps}\phi) \ind_{ \{ M^w(\sqrt{\eps} \phi) \in [D-r,D+r] \}   }  \Big\} \bigg] \notag \\
&=\sup_{u \in \Ha} \E\Bigg[ -V(\eps^\frac 12 W+Z(u) )\ind_{  \{ M^w(\eps^{\frac 12}W+Z(u) ) \in [D-r,D+r]  \}   } -\frac 12 \int_0^1 \| u(t)\|_{L^2_x(\R)}^2 dt   \Bigg] \notag \\
&\le \sup_{Z \in \mathbb{H}^1 } \E\bigg[ -V(\eps^\frac 12 W+Z )\ind_{  \{ M^w(\eps^{\frac 12}W+Z ) \in [D-r,D+r]  \}   } -\frac 12\| Z\|_{\mathcal{H}^1}^2   \bigg],
\label{A0}
\end{align}

\noi
where in the last line we used \eqref{LL3} and $\mathbb{H}^1$ represents the collection of drifts $Z$, characterized as processes that belong to $\H^1$ $\PP$-almost surely (possibly non-adapted). 

\noi 
In the following, we apply a change of variables to eliminate the Gaussian term $\eps^\frac 12 W$ as follows
\begin{align}
Z=-\eps^\frac 12 W_N+Q,
\label{cha0}
\end{align}

\noi
where $Q$ is in  $\mathbb{H}^1$ and $W_N=\P_N W$. In \eqref{cha0}, the Gaussian field $\eps^\frac 12 W_N$ can be regarded as a small-order fluctuation, and the typical configuration is determined by the drift $Q$. From \eqref{cha0} and  Young's inequality, we have
\begin{align}
&\frac \ld4 \int_{\R} | \eps^\frac12 (W-W_N) + Q|^4 dx \notag \\
&\le \ld C_\zeta \eps^2 \int_{\R} |W-W_N|^4 dx+\frac {\ld(1+\zeta)}4 \int_{\R} |Q|^4 dx
\label{A1}
\end{align}

\noi
for any $\zeta>0$, where $C_\zeta$  is a large constant arising from Young's inequality.

Regarding the entropy term, from \eqref{cha0} we obtain
\begin{align}
\frac 12 \| Z \|_{\mathcal{H}^1}^2&=\frac \eps2 \|   W_N\|_{\mathcal{H}^1}^2+\frac 12 \| Q\|_{\mathcal{H}^1}^2- \eps^\frac 12  \Re \int_{\R} \mathcal{L}^\frac 12 W_N \mathcal{L}^\frac 12 Q dx  \notag  \\
&\ge  \frac 12 (1-\eps^\frac 12) \| Q\|_{\mathcal{H}^1}^2- \frac{\eps^\frac 12}{4}  \|W_N \|_{\H^1}^2,
\label{A22}
\end{align}

\noi
which follows from choosing $\eps$  sufficiently small and Young's inequality
\begin{align*}
\eps^\frac 12 \bigg|\Re \int_{\R} \mathcal{L}^\frac 12 W_N \mathcal{L}^\frac 12 Q dx \bigg| \le \frac{\eps^\frac 12}2  \| W_N \|_{\mathcal{H}^1}^2+\frac{\eps^\frac 12}2  \|Q \|_{\mathcal{H}^1}^2. 
\end{align*}

\noi
Since $\E\big[ \|W_N \|_{\H^1}^2    \big]\sim O(N)$ as $N\to \infty$,  it follows from \eqref{A22} that
\begin{align}
\frac 12\E \big[  \| Z \|_{\mathcal{H}^1}^2 \big] \ge \frac 12 (1-\eps^\frac 12) \E\big[ \| Q\|_{\mathcal{H}^1}^2 \big]- C \eps^\frac 12 N 
\label{A2}
\end{align}

\noi
for some constant $C>0$.

We now consider the taming part. Expanding the renormalized $L^2$-norm yields 
\begin{align*}
M^w(\eps^\frac 12W+Z)&=\int_{\R} :\! |\eps^\frac 12W+Z|^2 \! : dx\\
&= \eps \int_{\R} :\! |W|^2 \!: dx+2\eps^\frac 12 \Re \int_{\R} WZ dx+\int_{\R} |Z|^2 dx.
\end{align*}

\noi 
From the change of variables $Z=-\eps^{\frac 12}W_N+Q,$ in \eqref{cha0}, we obtain
\begin{align}
M^w(\eps^\frac 12W+Z)&=\eps \int_{\R} :\! |W|^2 \!: dx-2\eps \Re \int_{\R} W\cdot W_N dx+\eps  \int_{\R} |W_N|^2 dx \notag \\
&\hphantom{X}+\int_{\R} |Q|^2 dx+2\eps^\frac 12 \Re \int_{\R} (W-W_N)Q dx \notag \\
&=\int_{\R} |Q|^2 dx+2\eps^\frac 12 \Re \int_{\R} (W-W_N)Q dx+ G(W_N,\eps),
\label{WL22}
\end{align}

\noi
where we split $M^w(\eps^\frac 12W+Z)$ into three parts  based on whether they depend on $\eps$ or $Q$. It follows from \eqref{WL22} that  
\begin{align}
A\bigg| \int_{\R}  :\! |\eps^\frac 12 W+Z|^2 \!: dx  \bigg|^3&=A\big| M^w(\eps^\frac 12W+Z) \big|^3 \notag \\
&=A \bigg| \int_\R |Q|^2 dx+2\eps^\frac 12 \Re \int_{\R} (W-W_N)Q dx+G(W_N,\eps) \bigg|^3 \notag \\
&\ge A(1-\zeta)\bigg( \int_{\R} |Q|^2 dx \bigg)^3 \notag \\
&\hphantom{X} -AC_\zeta \Bigg(\bigg|2\eps^\frac 12 \Re \int_{\R} (W-W_N)Q dx \bigg|^3+  |G(W_N, \eps) |^3 \Bigg),  
\label{A3}
\end{align}

\noi
where in the last step we used
\begin{align}
|a+b+c|^\g \ge (1-\zeta)|c|^\g-C_\zeta ( |a|^\g+|b|^\g)
\label{You1}
\end{align}

\noi
for any $a,b,c  \in \R$ and $\gamma > 0$, where $\zeta$ is an arbitrary small number and $C_\zeta$ is a sufficiently large corresponding  constant. Note that 
\begin{align}
\bigg| \int_{\R} (W-W_N)Q dx \bigg|^3 &\le \| W-W_N \|_{\H^{-\eta}}^3 \| Q \|_{\H^\eta}^3 \notag \\
&\le  \| W-W_N \|_{\H^{-\eta}}^3 \| Q \|_{L^2}^{3(1-\eta) } \|Q \|_{\H^1}^{3\eta },
\label{interp}
\end{align}

\noi 
where in the last step, we used the interpolation inequality
\begin{align*}
\| \phi \|_{\H^\eta}  \les \| \phi \|_{\mathcal{H}^1 }^\eta \| \phi \|_{L^2}^{1-\eta}.
\end{align*}

\noi
Since   $\frac{3(1-\eta)}{6}+\frac {3\eta}{2}<1$, \eqref{interp} and Young's inequality imply 
\begin{align}
A \eps^\frac 32 \bigg| \Re \int_{\R} (W-W_N)Q dx \bigg|^3 \les  A \eps^\frac 32  \| W-W_N \|_{\H^{-\eta}}^{c_0}+A \eps^\frac 32 \|Q \|_{L^2}^6+ A \eps^\frac 32 \| Q\|_{\H^1}^2.
\label{A4}
\end{align}

\noi
Combining \eqref{A3} and \eqref{A4} yields 
\begin{align}
A\bigg| \int_{\R}  :\! |\eps^\frac 12 W+Z|^2 \!: dx  \bigg|^3 \notag &\ge   A(1-\zeta-\eps^\frac 32)\bigg( \int_{\R} |Q|^2 dx \bigg)^3-A\eps^\frac 32 \| Q\|_{\mathcal{H}^1}^2 \notag \\
&\hphantom{X}-A \eps^\frac 32  \| W-W_N \|_{\H^{-\eta}}^{c_0}-AC_\zeta |G(W_N,\eps)|^3.
\label{A5}
\end{align}

\noi
Thanks to the fact that $W_N$ is a Gaussian field, we have $\E  \big[ \| W_N \|_{L^2(\R)}^6 \big] \les  \|  \big( \E |W_N|^6 \big)^\frac 16   \|_{L^2(\R)}^6  \les  \|  \big( \E |W_N|^2 \big)^\frac 12   \|_{L^2(\R)}^6 \sim (\log N)^3  $, which implies 
\begin{align}
\E \Big[ |G(W_N,\eps)|^3 \Big]  \les \eps^3 (\log N)^3
\label{A55}
\end{align}

\noi 
Combining \eqref{cha0}, \eqref{A1}, \eqref{A2}, \eqref{A5}, and \eqref{A55} yields 
\begin{align}
&\sup_{Z \in \mathbb{H}^1 } \E\bigg[ -V(\eps^\frac 12 W+Z )\ind_{  \{ M^w(\eps^{\frac 12}W+Z ) \in [D-r,D+r]  \}   } -\frac 12\| Z\|_{\mathcal{H}^1}^2   \bigg] \notag \\
&=\sup_{Q \in \mathbb{H}^1 } \E\bigg[ -V(\eps^\frac 12(W-W_N)+Q )\ind_{  \{ M^w(\eps^{\frac 12}(W-W_N)+Q ) \in [D-r,D+r]  \}   } -\frac 12\| \eps^\frac 12 W_N-Q \|_{\mathcal{H}^1}^2   \bigg]  \notag \\
&\le  \sup_{Q\in \mathbb{H}^1  } \E \Bigg[ \bigg( \frac{\ld(1+\zeta)}{4} \int_{\R} |Q|^4 dx  -A(1-\zeta-\eps^\frac 32)\bigg( \int_{\R} |Q|^2 dx \bigg)^3  -\frac 12 (1-\eps^\frac 12-2A\eps^\frac 32) \| Q\|_{\mathcal{H}^1}^2 \notag \\
&\hphantom{XXXXXXXXX}  \bigg)\cdot \ind_{  \{ M^w(\eps^{\frac 12}(W-W_N)+Q ) \in [D-r,D+r]  \}   } \Bigg] +C_\zeta \eps^2+C\eps^\frac 12 N+\eps^3 (\log N)^3 \notag \\
&=\sup_{Q\in \mathbb{H}^1  }\E\Big[ -H^G_{\zeta,\eps}(Q)\ind_{  \{ M^w(\eps^{\frac 12}(W-W_N)+Q ) \in [D-r,D+r]  \}   } \Big]+C_\zeta \eps^2+C\eps^\frac 12 N+\eps^3 (\log N)^3
\label{A66}
\end{align}

\noi
where 
\begin{align*}
H^{G}_{\zeta, \eps}(\phi)=\frac 12 (1-\eps^\frac 12-2A \eps^\frac 32 )\int_{\R} |\mathcal{L}^\frac 12 Q|^2 dx-\frac{\ld(1+\zeta)}{4} \int_{\R} |Q|^4 dx+A(1-\zeta-\eps^\frac 32) \bigg(\int_{R} |Q|^2 dx\bigg)^3.
\end{align*}

\noi
We will now handle the indicator $\mathbf{1}_{\{M^w(...) , \mathrm{etc.}\}}$ by separating the mass $M^w(\epsilon^{1/2} (W- W_N) + Q)$ into $M(Q) + $ (fluctuations depending on $\epsilon$). Using \eqref{WL22} together with the change of variables $Z=-\eps^{\frac 12}W_N+Q$ in \eqref{cha0}, we expand the renormalized $L^2$-norm 
\begin{align}
M^w(\eps^\frac 12W+Z)&=\int_{\R} |Q|^2 dx+\Phi(Q,W,\eps),
\label{A6}
\end{align}

\noi
where
\begin{align*}
\Phi(Q,W,\eps)=\eps \int_{\R} :\! |W|^2 \!: dx-2\eps \Re \int_{\R} W\cdot W_N dx+\eps  \int_{\R} |W_N|^2 dx +2\eps^\frac 12 \Re \int_{\R} (W-W_N)Q dx.
\end{align*}

\noi
From \eqref{A6},  we have 
\begin{align}
\big\{ M^w(\eps^{\frac 12}W+Z ) \in [D-r,D+r]  \big \}   \cap \big\{ |\Phi(W,Q,\eps)| < r  \big\}=\big \{  M(Q)  \in \big[D-2r, D+2r\big]   \big \},
\label{WL3}
\end{align}

\noi
where $M(Q)=\|Q \|_{L^2(\R)}^2$.  By splitting \eqref{A66} into the two cases  $\big\{ |\Phi(W,Q,\eps)| < r  \big\}$ and $\big\{ |\Phi(W,Q,\eps)| \ge r  \big\}$, and applying \eqref{WL3}, we obtain
\begin{align}
&\sup_{Z \in \mathbb{H}^1 } \E\bigg[ V(\eps^\frac 12 W+Z )\ind_{  \{ M^w(\eps^{\frac 12}W+Z ) \in [D-r,D+r]  \}   } -\frac 12\| Z\|_{\mathcal{H}^1}^2   \bigg] \notag  \\
& \le   \sup_{Q\in \mathbb{H}^1  }\E\Big[ -H^G_{\zeta,\eps}(Q)\ind_{  \{ M(Q ) \in [D-2r,D+2r]  \}   } \Big] \notag \\
&\hphantom{X}+\sup_{Q\in \mathbb{H}^1  }\E\Big[ -H^G_{\zeta,\eps}(Q) \ind_{  \{ M^w(\eps^{\frac 12}(W-W_N)+Q ) \in [D-r,D+r]  \}   }  \ind_{ \{ |\Phi(W,Q,\eps)| \ge r  \}      } \Big] \notag \\
&\hphantom{X}+C_\zeta \eps^2+C\eps^\frac 12 N+\eps^3 (\log N)^3 \notag \\
&\le  \sup_{Q\in \mathbb{H}^1  }\E\Big[ -H^G_{\zeta,\eps}(Q)\ind_{  \{ M(Q ) \in [D-2r,D+2r]  \}   } \Big]+C_\zeta \eps^2+C\eps^\frac 12 N+\eps^3 (\log N)^3,
\label{A7} 
\end{align}

\noi
where in the last line we used $H^G_{\zeta,\eps}(\phi) \ge 0 $ for any $\phi \in \H^1$
as long as the chemial potential $A$ is sufficiently large. See \eqref{coer0}.

It follows from \eqref{WL1}, \eqref{WI0}, \eqref{A0}, \eqref{A7}, and taking the limits $\eps \to 0, r\to 0$ that  
\begin{align*}
&\limsup_{r \to 0}\limsup_{\eps \to 0}\eps \log \rho_{\eps, A} \big( \{ M^w(\phi) \in [D-r,D+r] \} \big) \\
&\le 
\limsup_{r \to 0}\limsup_{\eps \to 0}\eps \log \E_{\mu}\bigg[ \exp\Big\{ -\frac 1\eps V(\sqrt{\eps}\phi) \ind_{ \{ M^w(\sqrt{\eps} \phi) \in [D-r,D+r] \}   }  \Big\} \bigg]\\
&  \le -\inf_{\substack{   M(\phi)=D}} H^G(\phi),
\end{align*}

\noi
where in the last step we also take the limit $\zeta \to 0$ since \eqref{A7} holds for any $\zeta>0$ arising from \eqref{A1} and \eqref{You1}. This completes the proof of the upper bound.

We now prove the lower bound 
\begin{align}
\liminf_{r \to 0}\liminf_{\eps \to 0}\eps \log \rho_{\eps, A} \big( \{ M^w(\phi) \in [D-r,D+r] \} \big) \ge -\inf\limits_{\substack{ M(\phi)=D}} H^G(\phi).
\label{C0}
\end{align}

\noi
Note that 
\begin{align*}
&\eps \log \rho_{\eps, A} \big( \{ M^w(\phi) \in [D-r,D+r] \} \big)\\
&=\eps \log \E_{\mu_\eps}\bigg[ \exp\Big\{ -\frac 1\eps V(\phi)  \Big\} \ind_{  \{ M^w(\phi) \in [D-r,D+r]  \}  } \bigg]-\eps \log Z_{\eps, A},
\end{align*}

\noi
where 
\begin{align*}
V(\phi)=-\frac \ld4\int_{\R}|\phi|^4 dx +A\bigg|\int_{\R} :\! |\phi|^2 \!: dx  \bigg|^3.
\end{align*}

\noi
By using the microcanonical condition $\{ M^w(\phi) \in [D-r,D+r]  \}$, we have 
\begin{align}
&\eps \log \rho_{\eps, A} \big( \{ M^w(\phi) \in [D-r,D+r] \} \big) \notag \\
&\ge \eps \log \E_{\mu_\eps}  \bigg[ \exp\Big\{-\frac 1\eps \mathcal{V} (\phi)  \Big\}  \ind_{ \{ M^w(\phi) \in [D-r,D+r]  \}  }  \bigg] -A(D+r)^3-\eps \log Z_{\eps,A},
\label{F0}
\end{align}

\noi
where
\begin{align*}
\mathcal{V}(\phi)=-\frac{\ld}4 \int_{\R} |\phi|^4 dx.
\end{align*}

\noi
Based on the following observation
\begin{align*}
&\E_{\mu_\eps}\bigg[ \exp\Big\{ -\frac 1\eps \mathcal{V}(\phi) \ind_{ \{ M^w(\phi) \in [D-r,D+r] \}   }  \Big\} \bigg]-1\\
&\le \E_{\mu_\eps}\bigg[ \exp\Big\{ -\frac 1\eps \mathcal{V}(\phi)   \Big\} \ind_{ \{ M^w(\phi) \in [D-r,D+r] \}   } \bigg]\\
&\le \E_{\mu_\eps}\bigg[ \exp\Big\{ -\frac 1\eps \mathcal{V}(\phi) \ind_{ \{ M^w(\phi) \in [D-r,D+r] \}   }  \Big\} \bigg], 
\end{align*}

\noi
we have 
\begin{align}
\log \wt Z_\eps +\log \Big(1-\frac 1{\wt Z_\eps}  \Big) \le \log Z_\eps \le \log \wt Z_\eps,
\label{C00}
\end{align}

\noi
where
\begin{align*}
\wt Z_\eps:&=\E_{\mu_\eps}\bigg[ \exp\Big\{ -\frac 1\eps \mathcal{V}(\phi) \ind_{ \{ M^w(\phi) \in [D-r,D+r] \}   }  \Big\} \bigg]\\
Z_\eps &=\E_{\mu_\eps}\bigg[ \exp\Big\{ -\frac 1\eps \mathcal{V}(\phi)  \Big\} \ind_{ \{ M^w(\phi) \in [D-r,D+r] \}   }  \bigg].
\end{align*}

\noi

\noi
In the following, we show that 
\begin{align}
\liminf_{r\to 0}\liminf_{\eps \to 0} \eps \log \wt Z_\eps \ge -\inf_{M(\phi)=D}H(\phi), 
\label{C0000}
\end{align}

\noi
which implies
\begin{align}
\wt Z_\eps \ges e^{\frac{c(D)}{\eps}} \to \infty 
\label{C000}
\end{align}

\noi
as $\eps \to 0$ and $r \to 0$, where $C(D)>0$ arises from the negative minimal energy
$\inf\limits_{M(\phi)=D}H(\phi)<0$ in \eqref{D0}. Combined with \eqref{C00} and \eqref{C000}, we obtain 
\begin{align}
\liminf_{r \to 0} \liminf_{\eps \to 0} \eps \log \wt Z_\eps=   \liminf_{r \to 0} \liminf_{\eps \to 0} \eps \log  Z_\eps.
\label{F1}
\end{align}

\noi
Therefore, the asymptotic behavior of $\eps \log Z_\eps$ can be obtained by studying $\eps \log \wt Z_\eps$. It follows from \eqref{F0}, \eqref{F1}, \eqref{C0000}, and \eqref{F00} that
\begin{align*}
&\liminf_{r\to 0}\liminf_{\eps \to 0}\eps \log \rho_{\eps, A} \big( \{ M^w(\phi) \in [D-r,D+r] \} \big)\\
&\ge \liminf_{r\to 0}\liminf_{\eps \to 0} \eps \log \E_{\mu_\eps}  \bigg[ \exp\Big\{-\frac 1\eps \mathcal{V} (\phi) \ind_{ \{ M^w(\phi) \in [D-r,D+r]  \}  }   \Big\}   \bigg] -AD^3\\
& \ge -\inf_{M(\phi)=D}H(\phi)-AD^3\\
&= -\inf_{M(\phi)=D}H^G(\phi).
\end{align*}

\noi
This shows the lower bound \eqref{C0}.  

It remains to prove  \eqref{C0000}. From \eqref{A0}, we write 
\begin{align}
&\eps \log \E_{\mu_\eps}\bigg[ \exp\Big\{ -\frac 1\eps \mathcal{V}(\phi) \ind_{ \{ M^w(\phi) \in [D-r,D+r] \}   }  \Big\} \bigg] \notag \\
&=\sup_{u \in \Ha}\E\bigg[ -\mathcal{V}(\eps^\frac 12 W+Z(u) )\ind_{  \{ M^w(\eps^{\frac 12}W+Z(u) ) \in [D-r,D+r]  \}   }  - \frac 12\int_0^1 \|u(t) \|_{L^2_x(\R)}^2 dt  \bigg].
\label{C1}
\end{align}

\noi
We choose a specific drfit $u \in \Ha$, defined by
\begin{align*}
u^0(t)=\frac{1}{\eta} \ind_{ \{  t> 1-\eta \} } \mathcal{L}^\frac 12(-\eps^\frac 12 W_N^0+Q),
\end{align*}

\noi
where 
\begin{align*}
W_N^0:=\sum_{|n|\le N} \frac{B_n(1-\eta)}{\ld_n} h_n(x)
\end{align*}

\noi
and $Q$ is the ground state with $L^2$ mass $\| Q\|_{L^2(\R)}^2=D$, namely the minimizer of $H$ in Lemma \ref{LEM:Min}.  Then, thanks to the cutoff $\ind_{ \{  t> 1-\eta \} }$ and the definition of $W_N^0$, the drift $u^0(t)$ belongs to the admissible class $\Ha$, being adapted to the filtration. From the definition of $Z(u)(t)$ in \eqref{defZu},
\begin{align}
Z(u)=Z(u)(1)=\int_0^1 \mathcal{L}^{-\frac 12} u^0(t) dt=-\eps^\frac 12 W_N^0+Q.
\label{C2}
\end{align}

\noi
Combining \eqref{C1} and \eqref{C2} yields 
\begin{align}
&\eps \log \E_{\mu_\eps}\bigg[ \exp\Big\{ -\frac 1\eps \mathcal{V}(\phi) \ind_{ \{ M^w(\phi) \in [D-r,D+r] \}   }  \Big\} \bigg] \notag \\
&\ge  \E\bigg[ -\mathcal{V}(\eps^\frac 12 (W-W_N^0)+Q )\ind_{  \{ M^w(\eps^\frac 12 (W-W_N^0)+Q  ) \in [D-r,D+r]  \}   } -\frac \eps2 \| W_N^0\|^2_{\mathcal{H}^1} -\frac 12 \|Q \|_{\mathcal{H}^1}^2 \notag \\
&\hphantom{XXXX}- \eps^\frac 12\Re \int_{\R} \mathcal{L}^\frac 12 W_N^0 \mathcal{L}^\frac 12 Q dx      \bigg] \notag \\
&\ge  \E\bigg[ -\mathcal{V}(\eps^\frac 12 (W-W_N^0)+Q )\ind_{  \{ M^w(\eps^\frac 12 (W-W_N^0)+Q  ) \in [D-r,D+r]  \}   } -\frac {\eps+\eps^\frac 12}2 \| W_N^0\|^2_{\mathcal{H}^1} -\frac {1+\eps^\frac 12}2 \|Q \|_{\mathcal{H}^1}^2  \bigg] \notag \\
& \ge \E\bigg[ -\mathcal{V}(\eps^\frac 12 (W-W_N^0)+Q )\ind_{  \{ M^w(\eps^\frac 12 (W-W_N^0)+Q  ) \in [D-r,D+r]  \}   } \bigg]- C \eps^\frac 12  N-\frac {1+\eps^\frac 12}2 \|Q \|_{\mathcal{H}^1}^2
\label{C22}
\end{align}

\noi
for some constant $C>0$, where in the third line and the last line, we used Young's inequality and $\E\big[ \| W_N^0 \|_{L^2}^2\big] \sim N$ as $N \to \infty$, respectively.

Recall  from \eqref{A6} and \eqref{C2} that 
\begin{align}
M^w(\eps^\frac 12 (W-W_N^0)+Q  )=\int_{\R} |Q|^2 dx+\Phi(Q,W_N^0,\eps)
\label{C4}
\end{align}

\noi
where
\begin{align*}
\Phi(Q,W_N^0,\eps)=\eps \int_{\R} :\! |W|^2 \!: dx-2\eps \Re \int_{\R} W\cdot W_N^0 dx+\eps  \int_{\R} |W_N^0|^2 dx +2\eps^\frac 12 \Re \int_{\R} (W-W_N^0)Q dx.
\end{align*}

\noi
Combining $\eqref{C4}$ and $\|Q\|_{L^2(\R)}^2=D$ yields 
\begin{align*}
\big\{ M^w(\eps^\frac 12 (W-W_N^0)+Q  ) \in [D-r,D+r]  \big \} =\{ |\Phi(Q,W_N^0,\eps)|  \le r \},
\end{align*}

\noi
which implies 
\begin{align}
\PP\big\{ M^w(\eps^\frac 12 (W-W_N^0)+Q  )  \in [D-r,D+r]    \big \}&=1- \PP\{ |\Phi(Q,W_N^0,\eps)|  > r \} \notag \\
&\ge 1-\frac{ \E \big[ |\Phi(Q,W_N^0,\eps)|^2 \big]}{r^2}.
\label{C5}
\end{align}

\noi
A similar calculation to \eqref{A55} gives
\begin{align}
\E \big[ |\Phi(Q,W_N^0,\eps)|^2 \big] \les \eps^2 (\log N)^2.
\label{C6}
\end{align}

\noi
It follows from \eqref{C5} and \eqref{C6} that 
\begin{align}
\PP\big\{ M^w(\eps^\frac 12 (W-W_N^0)+Q  )  \in [D-r,D+r]    \big \} \ge 1 -r^{-2}\eps^2 (\log N)^2.
\label{C55}
\end{align}

\noi
We now consider the quartic interaction part. Using the elementary inequality \eqref{You1} with $\gamma = 4$,
\begin{align}
|a+b+c|^4 \ge (1-\zeta)|c|^4-C_\zeta(|a|^4+|b|^4)
\label{You11}
\end{align}

\noi
we can use \eqref{You11} and \eqref{C55} to find 
\begin{align}
&\E\Bigg[ \bigg(\int_{\R}| \eps^\frac 12 (W-W_N^0)+Q     |^4 dx\bigg)\ind_{  \{ M^w(\eps^\frac 12 (W-W_N^0)+Q  ) \in [D-r,D+r]  \}   }    \Bigg] \notag \\
&\ge \E\Bigg[(1-\zeta)^4 \bigg(\int_{\R} |Q|^4 dx\bigg)\ind_{  \{ M^w(\eps^\frac 12 (W-W_N^0)+Q  ) \in [D-r,D+r]  \}   }  -C_\zeta \int_{\R} |\eps^{\frac 12}(W-W_N^0) |^4 dx   \Bigg] \notag \\
& \ge (1-\zeta)^4  (1 -r^{-2}\eps^2 (\log N)^2) \bigg(\int_{\R} |Q|^4 dx\bigg)- \wt C_\zeta \eps^2
\label{C77}
\end{align} 

\noi
for some large constant $\wt C_\zeta>0$.  Combining the lower bounds from \eqref{C22} and \eqref{C77}, we have
\begin{align*}
&\eps \log \E_{\mu_\eps}\bigg[ \exp\Big\{ -\frac 1\eps V(\phi) \ind_{ \{ M^w(\phi) \in [D-r,D+r] \}   }  \Big\} \bigg]\\
&\ge \frac \ld4 (1-\zeta)^4  (1 -r^{-2}\eps^2 (\log N)^2) \bigg(\int_{\R} |Q|^4 dx\bigg) -\frac {1+\eps^\frac 12}2 \|Q \|_{\mathcal{H}^1}^2\\
&\hphantom{XX} - \wt C_\zeta \eps^2-AC_\zeta \eps^3 (\log N)^3-C\eps^\frac 12 N.
\end{align*}

\noi 
Taking first the limit $\eps \to 0$, then $r \to 0$, and finally $\zeta \to 0$, we obtain  
\begin{align*}
&\liminf_{r\to 0}\liminf_{\eps \to 0}\eps \log \E_{\mu_\eps}\bigg[ \exp\Big\{ -\frac 1\eps V(\phi) \ind_{ \{ M^w(\phi) \in [D-r,D+r] \}   }  \Big\} \bigg]\\
& \ge  \frac \ld4 \int_{\R} |Q|^4 dx  -\frac 12 \int_{\R} |\nb Q|^2 dx- \int_{\R} |x|^2|Q|^2 dx\\
&= -\inf_{M(\phi)=D} H(\phi).
\end{align*}

\noi 
This completes the proof of \eqref{C0000}.

\end{proof}


\begin{remark}\rm \label{REM:wcond}
It follows from Proposition \ref{PROP:WL2} that
\begin{align*}
\lim_{r \to 0}\lim_{\eps \to 0}\eps \log \rho_{\eps, A} \big( \{ M^w(\phi) \in [D-r,D+r] \} \big)=-\inf\limits_{\substack{ M(\phi)=D}} H^G(\phi)>-\infty.
\end{align*}

\noi
This implies that 
\begin{align*}
\rho_{\eps, A} \big( \{ M^w(\phi) \in [D-r,D+r] \} \big)>0
\end{align*}

\noi
for sufficiently small $\eps>0$ and $r>0$.  Therefore, the conditional probability measure $\rho_{\eps,r}^D$ in \eqref{Gibbs5} is well defined.
\end{remark}

In the following, we describe the asymptotic behavior of the free energy in the low temperature limit.
\begin{proposition}\label{PROP:free}
Let $A \ge A_0$, where $A_0$ is given in Lemma \ref{LEM:coercive}. Then,
\begin{align*}
\lim_{\eps \to 0}\eps \log Z_{\eps, A}=-\inf_{\phi \in \mathcal{H}^{1} }H^G(\phi),
\end{align*}

\noi
where $Z_{\eps,A}$ is the grand canonical partition function defined in \eqref{Gibbs5}
and  $H^G$ is the grand canonical Hamiltonian $H^G$ given in \eqref{Ha1}.

\end{proposition}

\begin{proof}
We can follow the proof of Proposition \ref{PROP:WL2} without the constraint $\{M^w(\phi) \in [D-r,D_r] \}$, which makes the argument much simpler.
\end{proof}

\section{Proof of the LDP for the mixed ensemble}\label{SEC:LDP}

In this section, we present the proofs of Theorems~\ref{THM:1} and \ref{THM:2}.
We first show  that the mixed ensembles $\{\rho^D_{\eps,r} \}_{\eps,r}$, defined in \eqref{Gibbs4}, satisfy a large deviation principle with the rate function $J^D$ 
\begin{equation}\label{rateD}
\begin{split}
J^D(\phi)=
\begin{cases}
H(\phi)-\inf\limits_{M(\phi)=D} H(\phi) \quad &\text{if} \quad  \phi\in \mathcal{H}^1(\R) \quad \text{and} \quad M(\phi)=D \\
\infty \quad &\text{otherwise}.
\end{cases}
\end{split}
\end{equation}

\noi 
and speed $\eps>0$. In other words,
\begin{itemize}
\item[(1)] For every closed set $\mathcal{C} \subset \mathcal{S}=\mathcal{H}^{-\eta}(\R)$ or $L^p(\R)$, $p>2$, we have 
\begin{align}
\limsup_{r \to 0}\limsup_{\eps \to 0} \eps \log \rho_{\eps,r}^D(\mathcal{C})&\le -\inf_{\phi \in \mathcal{C}} J^D(\phi).
\label{LDPU}
\end{align}

\vspace{2mm}

\item[(2)] For every open set $\mathcal{O}\subset \mathcal{S}=\mathcal{H}^{-\eta}(\R)$ or $L^p(\R)$, $p>2$, we have 
\begin{align}
\liminf_{r\to 0}\liminf_{\eps \to 0} \eps \log \rho_{\eps,r}^D(\mathcal{O})&\ge -\inf_{\phi \in \mathcal{O} } J^D(\phi).
\label{LDPL}
\end{align} 

\end{itemize}

\begin{proof}[Proof of Theorem \ref{THM:1}]
To prove the large deviation upper bound in \eqref{LDPU}, we first show that for any given $\zeta>0$, there exists $\dl>0$ such that 
\begin{align}
\limsup_{r\to 0}\limsup_{\eps \to 0} \eps \log \rho_{\eps, r}^D (\{ \phi \in \cj B(\psi,\dl) \}) \le -J^D(\psi)+\zeta,
\label{G1}
\end{align}

\noi
where $B(\psi, \dl)$ denotes the open ball with center $\psi$ and radius $\dl>0$ with respect to $\mathcal{S}=\mathcal{H}^{-\eta}(\R)$ or $L^p(\R)$, $p>2$. Using the large deviation upper bound for the grand-canonical ensemble $\rho_{\eps,A}$ in Proposition \ref{PROP:GLDP}, we have 
\begin{align}
&\limsup_{\eps \to 0} \eps \log \rho_{\eps, r}^D (\{ \phi \in \cj B(\psi,\dl) \}) \notag \\
& \le \limsup_{\eps \to 0 } \eps \log \rho_{\eps,A} \big(  \cj B(\psi,\dl)  \cap \{M^w(\phi)\in [D-r,D+r]     \} \big) \notag  \\
&\hphantom{XX}- \liminf_{\eps \to 0} \eps \log \rho_{\eps,A}  \big( \{ M^w(\phi) \in [D-r,D+r]  \} \big)   \notag \\
&\le -\inf_{\phi \in \cj B(\psi,\dl)} J^G(\phi)- \liminf_{\eps \to 0} \eps \log \rho_{\eps,A}  \big( \{ M^w(\phi) \in [D-r,D+r]  \} \big).
\label{B1}
\end{align} 

\noi
Using the lower semicontinuity of $J^G$, we obtain that for any given $\zeta>0$, there exists $\dl>0$ such that 
\begin{align}
\inf\limits_{\phi \in \cj B(\psi,\dl) } J^G(\phi) \ge J^G(\psi)-\zeta.
\label{B2}
\end{align}

\noi
It follows from \eqref{B1}, \eqref{B2}, taking the limit $r\to 0$, and Proposition \ref{PROP:WL2} that   
\begin{align}
&\limsup_{r\to 0}\limsup_{\eps \to 0} \eps \log \rho_{\eps, r}^D (\{ \phi \in \cj B(\psi,\dl) \}) \notag \\
&\le -J^G(\psi)+\inf_{ M(\phi)=D} H^G(\phi)+\zeta.
\label{BB2}
\end{align}

\noi
We first consider the case $\psi \in \H^{1}$ with $M(\psi)=D$. By using the fact that $\psi\in \H^1$ and $M(\psi)=D$, we obtain
\begin{align}
&-J^G(\psi)+\inf_{ M(\phi)=D} H^G(\phi)+\zeta \notag \\
&=-H(\psi)+\inf_{M(\phi)=D} H(\phi)+\zeta \notag \\
&=-J^D(\psi)+\zeta,
\label{BB3}
\end{align}

\noi
Combining \eqref{BB2} and \eqref{BB3} yields
\begin{align*}
\limsup_{r\to 0}\limsup_{\eps \to 0} \eps \log \rho_{\eps, r}^D (\{ \phi \in \cj B(\psi,\dl) \})  \le -J^D(\psi)+\zeta,
\end{align*}

\noi
which completes  the proof of \eqref{G1} in the case $\psi\in \H^1$ and $M(\psi)=D$.

We now consider the case $\psi \notin \H^1$. By the definitions of $J^G$ and $J^D$ in \eqref{rateG} and \eqref{rateD}, it follows that $J^G(\psi)=\infty$ and $J^D(\psi)=\infty$. Therefore, \eqref{BB2} implies 
\begin{align*}
\limsup_{r\to 0}\limsup_{\eps \to 0} \eps \log \rho_{\eps, r}^D (\{ \phi \in \cj B(\psi,\dl) \}) \le  -\infty= -J^D(\psi), 
\end{align*}

\noi 
which completes  the proof of \eqref{G1} in the case $\psi \notin \H^1$.

\noi 
We now consider the case where $\psi \in \H^1$ and $M(\psi)\neq D$. Then, by the definition of $J^D$, we have $J^D(\psi)=\infty$. Note that 
\begin{align}
&\limsup_{\eps \to 0} \eps \log \rho_{\eps, r}^D (\{ \phi \in \cj B(\psi,\dl) \}) \notag \\
& = \limsup_{\eps \to 0 } \eps \log \rho_{\eps,A} \big(  \cj B(\psi,\dl)  \cap \{M^w(\phi)\in [D-r,D+r]     \} \big) \notag \\
&\hphantom{XX}- \liminf_{\eps \to 0} \eps \log \rho_{\eps,A}  \big( \{ M^w(\phi) \in [D-r,D+r]  \} \big).  
\label{B3}
\end{align}

\noi
By following the proof of Proposition \ref{PROP:WL2}, we obtain  
\begin{align}
&\lim_{r\to 0}\lim_{\eps \to 0 } \eps \log \rho_{\eps,A} \big(  \cj B(\psi,\dl)  \cap \{M^w(\phi)\in [D-r,D+r]     \} \big)\notag \\
& = -\inf_{\substack{\phi \in \H^1 \\ M(\phi)=D \\ \| \phi-\psi \|_{\mathcal{S} }<\frac \dl2  } } H^G(\phi)
\label{B4}
\end{align}

\noi
for any $\dl>0$. Since $\psi \in L^2 $ and $M(\psi)\neq D$,
the conditons $M(\phi)=D$ and $\| \phi-\psi \|_{\mathcal{S}}<\frac \dl2$ for arbitrary small $\dl>0$ lead to a contradiction,  implying that the infimum in \eqref{B4} is taken over the empty set. Therefore,
\begin{align}
\inf_{\substack{\phi \in \H^1 \\ M(\phi)=D \\ \| \phi-\psi \|_{\H^{-\eta}}<\frac \dl2  } } H^G(\phi)=\infty.
\label{BB4}
\end{align}

\noi 
Combining \eqref{B3}, \eqref{B4}, \eqref{BB4}, and Proposition \ref{PROP:WL2} yields 
\begin{align*}
&\limsup_{ r\to0}\limsup_{\eps \to 0} \eps \log \rho_{\eps, r}^D (\{ \phi \in \cj B(\psi,\dl) \})\\
&\le-\infty=-J^D(\psi).
\end{align*}

\noi
This completes the proof of \eqref{G1} in the case where $\psi \in \H^1$ and  $M(\psi) \neq D$.

\noi
Take any compact set $K$ in $H^{-\eta}$. By compactness, we can cover $K$ with finitely many closed balls $ \cj B(\psi_j, \dl_j)$ centered at  $\psi_j \in K$ with small $\dl_j>0$ and  
\begin{align}
\limsup_{r\to 0}\limsup_{\eps \to 0} \eps \log \rho_{\eps, r}^D (\{ \phi \in \cj B(\psi_j,\dl_j) \}) \le -J^D(\psi_j)+\zeta,
\label{B8}
\end{align}

\noi
which follows from \eqref{G1}. Recall that for a collection of sequences $\{x_n^j\}$ of positive real numbers, indexed by $1\le j \le N$ for some finite $N$, we have 
\begin{align}
\limsup_{n \to \infty}\frac 1n \log\Big(\sum_{j=1}^Nx_n^j\Big) =\max_{1\le j \le N} \big\{  \limsup_{n\to \infty} \frac 1n \log x_n^j \big\}.
\label{B9} 
\end{align}

\noi 
Combining \eqref{B8} and \eqref{B9} yields 
\begin{align*}
\limsup_{r \to 0} \limsup_{ \eps \to 0} \eps \log \rho_{\eps,r}^D(\{ \phi \in K  \}) \le -\min_j J^D(\psi_j)+\zeta \le -\inf_{\phi \in K} J^D(\phi)+\zeta. 
\end{align*}

\noi
By taking $\zeta \to 0$, we obtain the large deviation upper bound  \eqref{LDPU} for any compact set $K$. To extend the upper bound from any compact set to any closed set, see \cite[Lemma 3.3]{EHT}.

Next, we prove the large deviation lower bound in \eqref{LDPL}. In the following we show that for any open set $G$
\begin{align}
\liminf_{r \to 0} \liminf_{ \eps \to 0} \eps \log \rho_{\eps,r}^D(\{ \phi \in G  \}) \ge -\inf_{\phi \in G} J^D(\phi).
\label{B10}
\end{align}

Take any $\psi \in G$.  By choosing $\dl>0$ sufficiently small, we have $B(\psi,\dl) \subset G$. Note that 
\begin{align}
&\liminf_{\eps \to 0} \eps \log  \rho_{\eps,r}^D(\{ \phi \in G  \})  \notag \\
 &\ge \liminf_{\eps \to 0} \eps \log \rho_{\eps,r}^D (  \{  \phi \in B(\psi,\dl )  \} ) \notag \\
& \ge \liminf_{\eps \to 0 } \eps \log \rho_{\eps,A} \big(  B(\psi,\dl)  \cap \{M^w(\phi)\in [D-r,D+r]     \} \big) \notag  \\
&\hphantom{X}- \limsup_{\eps \to 0} \eps \log \rho_{\eps,A}  \big( \{ M^w(\phi) \in [D-r,D+r]  \} \big).
\label{B5}
\end{align}

\noi 
By following the proof of Proposition \ref{PROP:WL2}, we obtain  
\begin{align}
&\lim_{r\to 0}\lim_{\eps \to 0 } \eps \log \rho_{\eps,A} \big(  \cj B(\psi,\dl)  \cap \{M^w(\phi)\in [D-r,D+r]     \} \big)\notag \\
& = -\inf_{\substack{\phi \in \H^1 \\ M(\phi)=D \\ \| \phi-\psi \|_{\mathcal{S} }<\frac \dl2  } } H^G(\phi) 
\label{B6}
\end{align}

\noi
for any $\dl>0$ and any $\eta>0$. Therefore, \eqref{B5}, \eqref{B6}, and Proposition \ref{PROP:WL2} imply  
\begin{align}
&\liminf_{r \to 0} \liminf_{\eps \to 0} \eps \log  \rho_{\eps,r}^D(\{ \phi \in G  \}) \notag  \\
& \ge -\inf_{\substack{\phi \in \H^1 \\ M(\phi)=D \\ \| \phi-\psi \|_{\mathcal{S} }<\frac \dl2  } } H^G(\phi) +\inf_{ M(\phi)=D} H^G(\phi)
\label{BB7} 
\end{align}

\noi
for any $\dl>0$. We first consider the case where $\psi \in \H^1$ and  $M(\psi)=D$. Since $M(\psi)=D$, we have 
\begin{align}
\inf_{\substack{\phi \in \H^1 \\ M(\phi)=D \\ \| \phi-\psi \|_{\mathcal{S}}<\frac \dl2  } } H^G(\phi) \le H^G(\psi).
\label{B7}
\end{align}

\noi 
It follows from \eqref{BB7} and \eqref{B7}  that   
\begin{align}
&\liminf_{r\to 0}\liminf_{\eps \to 0} \eps \log  \rho_{\eps,r}^D(\{ \phi \in G  \}) \notag  \\
&\ge - H^G(\psi)+\inf_{M(\phi)=D} H^G(\phi) \notag \\
&=-H(\psi)+\inf_{M(\phi)=D} H(\phi)=-J^D(\psi),
\label{B11}
\end{align}

\noi 
where in the last two steps, we used $\psi\in \H^1$ and $M(\psi)=D$.

We now consider the case where either $\psi \notin \H^1$ or $M(\psi)\neq D$. In either scenario, by the definition of the rate function $J^D$ in \eqref{rateD}, we have $J^D(\psi)=\infty$. Consequently, we obtain
\begin{align}
\liminf_{r\to 0}\liminf_{\eps \to 0} \eps \log  \rho_{\eps,r}^D(\{ \phi \in G  \}) &\ge  -\infty =-J^D(\psi).
\label{B12}
\end{align}

From \eqref{B11} and \eqref{B12}, we obtain 
\begin{align*}
\liminf_{r\to 0}\liminf_{\eps \to 0} \eps \log  \rho_{\eps,r}^D(\{ \phi \in G  \})
&\ge -J^D(\psi)
\end{align*}

\noi
for any $\psi \in G$, which implies that 
\begin{align*}
\liminf_{r\to 0}\liminf_{\eps \to 0} \eps \log  \rho_{\eps,r}^D(\{ \phi \in G  \})  \ge \sup_{\psi \in G} \{  -J^D(\psi)  \}=-\inf_{\psi \in G} J^D(\psi).
\end{align*}

\noi
This completes the proof of the large deviation lower bound in \eqref{LDPL}.

\end{proof}

Before presenting the proof of Theorem~\ref{THM:2}, we first prove the following lemma. This lemma shows that if $\phi$ is far from the family $\M_D$ of minimizers for \eqref{Min0}, then $H(\phi)$ is also far from the minimal energy.
\begin{lemma}\label{LEM:sta}
Let $2\le p<\infty$. For every $\dl>0$, there exists $c(\dl)>0$ such that if $\phi \in \mathcal{H}^1(\R)$ satisfies $\|  \phi\|_{L^2(\R)}^2=D$ and
\begin{align*}
\inf_{Q \in \M_D}\| \phi-Q \|_{L^p(\R)} \ge \dl,
\end{align*}

\noi
then  
\begin{align*}
H(\phi) \ge \inf_{Q \in \M_D} H(Q)+c(\dl),
\end{align*}

\noi
where $\M_D$  denotes the set of minimizers for \eqref{Min0}

\end{lemma}

\begin{proof}
To prove the statement, we proceed by contradiction. That is, supposed that there exists $\dl>0$ such that for every $n \ge 1$, there exists $\phi_n$ satisfying $M(\phi_n)=D$ and 
\begin{align}
\inf_{Q \in \M_D}\| \phi_n-Q \|_{L^p(\R)} \ge \dl,
\label{SS0}
\end{align}

\noi
but 
\begin{align*}
H(\phi_n) < \inf_{Q \in \M_D} H(Q)+\frac 1n.
\end{align*}

\noi
This implies that $H(\phi_n) \to \inf\limits_{Q \in \M_D} H(Q)$ as $n\to \infty$. That is, $\{\phi_n\}_{n \ge 1}$ is a minimizing sequence. The Gagliardo–Nirenberg–Sobolev inequality~\eqref{GNS} implies that
\begin{align*}
&H(\phi_n)=\frac 12 \int_{\R} |\dx \phi_n|^2 +\frac 12 \int_{\R} |x|^2 |\phi_n|^2 dx-\frac \ld4 \int_{\R} |\phi_n|^4 dx\\
&\ge \frac 12 \int_{\R} |\dx \phi_n|^2 +\frac 12 \int_{\R} |x|^2 |\phi_n|^2 dx -C\frac \ld 4 \| \phi_n \|_{\mathcal{H}^1(\R)} \| \phi_n \|_{L^2(\R)}^3\\
&=\frac 12 \| \phi_n \|_{\mathcal{H}^1(\R)}^2-C\frac \ld 4 \| \phi_n \|_{\mathcal{H}^1(\R)}D^\frac 32,
\end{align*}

\noi
where $C>0$ comes from the  Gagliardo–Nirenberg–Sobolev inequality \eqref{GNS}. Combined with $H(\phi_n) \to \inf\limits_{Q \in \M_D} H(Q)<\infty$ as $n \to \infty$, we conclude that $\{\phi_n \}_{ n \ge 1}$ is bounded in $\mathcal{H}^1(\R)$. Therefore, there exists $\psi \in \mathcal{H}^1(\R)$ such that $\phi_n $ converges weakly to $\psi$ in $\mathcal{H}^1(\R)$. It follows from \cite[Lemma 3.1]{Zhang} that the embedding $\mathcal{H}^1(\R) \hookrightarrow L^p(\R)$ is compact for any $2\le p<\infty$. This implies that up to a subsequence, 
\begin{align*}
\| \phi_n-\psi \|_{L^p(\R)} \to 0
\end{align*}

\noi
as $n \to \infty$ for any $2\le p<\infty$. Since $M(\phi_n)=D$ for every $n \ge 1$, we obtain $M(\psi)=D$. In other words, there is no escape of mass. Hence, we can conclude that $\psi$ is a minimizer for \eqref{Min0}. Therefore,
\begin{align*}
\inf_{Q \in \M_D}\| \phi_n-Q \|_{L^p(\R)} \le  \| \phi_n-\psi \|_{L^p(\R)} \to 0 
\end{align*}

\noi
as $n \to \infty$, which contradicts \eqref{SS0}. This completes the proof of Lemma \ref{LEM:sta}.
\end{proof}

We are now ready to present the proof of Theorem \ref{THM:2}.
\begin{proof}
It follows from Lemma  \ref{LEM:sta} that if
\begin{align*}
\inf_{Q\in \M^D} \|  \phi-Q    \|_{L^p(\R) } \ge \dl, 
\end{align*}

\noi
then we have 
\begin{align*}
J^D(\phi) \ge c(\dl)>0
\end{align*}

\noi
for some $c(\dl)>0$ arising from Lemma \ref{LEM:sta}. Combined with Theorem \ref{THM:1}, we obtain that  for any $2<p<\infty$,  
\begin{align*}
\rho_{\eps,r}^D\bigg( \Big\{ \inf_{Q\in \M^D} \|  \phi-Q    \|_{L^p(\R) } \ge \dl   \Big\}  \bigg) \les e^{-\frac{c(\dl)}{\eps}}
\end{align*}

\noi
as $\eps \to 0$ and $r \to 0$. This completes the proof of Theorem \ref{THM:2}.

\end{proof}

\appendix

\section{Variational characterization of the grand canonical Gibbs ensemble}

In this appendix, we study the Laplace integral with respect to the grand canonical Gibbs ensemble $\rho_{\eps,A}$ \eqref{Gibbs5} and its variational representation. 
For any  continuous and bounded functional $f: \mathcal{S}'(\R) \to \R$, 
\begin{align*}
-\eps \log \int  e^{-\frac 1\eps f(\phi)}\rho_{\eps,A}(d\phi)=-\eps \log \E_{\mu_\eps}\Big[e^{-\frac 1\eps (f(\phi)+V(\phi) ) }  \Big]+\eps \log \E_{\mu_\eps } \Big[ e^{-\frac 1\eps  V(\phi)  }    \Big],
\end{align*}

\noi
where $V$ is defined in \eqref{potent}. If $\phi$ represents a Gaussian random field with $\Law(\phi)=\mu=\mu_1$ whose covariance is $\mathcal{L}^{-1}$, applying the linear transformation $\phi \mapsto \sqrt{\eps} \phi$, $\sqrt{\eps} \phi$ yields a
Gaussian random field with $\Law(\sqrt{\eps} \phi )=\mu_\eps$ whose covariance is $\eps \mathcal{L}^{-1}$. Therefore, 
\begin{align}
-\eps \log \int  e^{-\frac 1\eps f(\phi)}\rho_{\eps,A}(d\phi)=-\eps \log \E_{\mu}\Big[e^{-\frac 1\eps (f( \sqrt{\eps} \phi)+V( \sqrt{\eps}\phi) ) }  \Big]+\eps \log \E_{\mu } \Big[ e^{-\frac 1\eps  V(\sqrt{\eps} \phi)  }    \Big].
\label{APE1}
\end{align}

\noi
In the following, we present the variational representation of the right hand side of \eqref{APE1}.

\begin{lemma}\label{LEM:Gama}
Let $f: \mathcal{S}'(\R) \to \R$  be a continuous and bounded functional. Then,
\begin{align*}
&-\eps \log \E_{\mu}\Big[e^{-\frac 1\eps (f( \sqrt{\eps} \phi)+V( \sqrt{\eps}\phi) ) }  \Big]\\
&=\inf_{u \in \Ha } \E\bigg[f(\eps^{\frac 12}W+ \eps^\frac 12Z(u) )+ V(\eps^{\frac 12}W+  \eps^\frac 12 Z(u) )+\frac \eps 2 \int_0^1 \| u(t) \|_{L^2(\R)}^2  dt  \bigg].
\end{align*}
\end{lemma}

\begin{proof}
By applying the Bou\'e-Dupuis formula (Lemma \ref{LEM:var}), we obtain that for any ultraviolet cutoff $\P_N$, $N \ge 1$ 
\begin{align}
&-\eps \log \E_{\mu}\Big[e^{-\frac 1\eps (f( \sqrt{\eps} \phi_N)+V( \sqrt{\eps}\phi_N) ) }  \Big] \notag \\
&=\inf_{u \in \Ha } \E\bigg[f(\eps^{\frac 12}W_N+ \eps^\frac 12 \P_N Z(u) )+ V(\eps^{\frac 12}W_N+  \eps^\frac 12 \P_N Z(u) )+\frac \eps 2 \int_0^1 \| u(t) \|_{L^2(\R)}^2  dt  \bigg].
\label{APE2}
\end{align}

\noi
Following the proof in \cite[Theorem 1]{BG}, specifically the Gamma convergence of \eqref{APE2}
as $N \to \infty$, we obtain 
\begin{align*}
&\lim_{N\to \infty}\inf_{u \in \Ha } \E\bigg[f(\eps^{\frac 12}W_N+ \eps^\frac 12 \P_N Z(u) )+ V(\eps^{\frac 12}W_N+  \eps^\frac 12 \P_N Z(u) )+\frac \eps 2 \int_0^1 \| u(t) \|_{L^2(\R)}^2  dt  \bigg]\\
&= \inf_{u \in \Ha } \E\bigg[f(\eps^{\frac 12}W+ \eps^\frac 12  Z(u) )+ V(\eps^{\frac 12}W+  \eps^\frac 12  Z(u) )+\frac \eps 2 \int_0^1 \| u(t) \|_{L^2(\R)}^2  dt  \bigg].
\end{align*}

\noi
This completes the proof of Lemma \ref{LEM:Gama}.

\end{proof}

Applying Lemma \ref{LEM:Gama} and the change of variables $\eps^\frac 12 u \to u$, 
\begin{align*}
-\eps \log \int  e^{-\frac 1\eps f(\phi)}\rho_{\eps,A}(d\phi)= \inf_{u\in \Ha}\mathcal{F}^{V+f, \eps}(u)- \inf_{u\in \Ha}\mathcal{F}^{V,\eps}(u),
\end{align*}

\noi
where 
\begin{align}
\mathcal{F}^{V+f, \eps}(u)=\E_{\PP}\bigg[f(\eps^{\frac 12}W+ Z(u) )+ V(\eps^{\frac 12}W+ Z(u) )+\frac 12 \int_0^1 \| u(t) \|_{L^2(\R)}^2 dt   \bigg].
\label{APE3}
\end{align}

\noi
By expanding the potential $V$ in \eqref{potent}, we can rewrite the variational problem \eqref{APE3} in the following form.

\begin{lemma}\label{LEM:Gamma1}
Let $f: \mathcal{S}'(\R) \to \R$  be a continuous and bounded functional. Then,
\begin{align}
\inf_{u\in \Ha} \mathcal{F}^{f+V,\eps}(u)&=\inf_{u \in \Ha } \E_{\PP} \bigg[f(\eps^\frac 12 W+Z(u))+\mathcal{E}(\mathbb{W}, Z(u),\eps )-\frac\ld4 \int_{\R} |Z(u)|^4 dx \notag \\
&\hphantom{XXXXXXXXXXXXX}+A\bigg( \int_{\R} |Z(u)|^2 dx \bigg)^3 + \frac 12 \int_0^1 \| u(t) \|^2_{L^2(\R)} dt \bigg].
\label{APE4}
\end{align}

\noi
where $\mathbb{W}=(W, \, :\! |W|^2 \!: )$ and $\ld>0$ is a coupling constant. Here, 
\begin{align}
\mathcal{E}(\mathbb{W}, Z(u),\eps )=\sum_{j=1}^2 \mathcal{E}_j(\mathbb{W}, Z(u),\eps ),
\label{AP1} 
\end{align}

\noi 
where  
\begin{align*}
\mathcal{E}_1(\mathbb{W}, Z(u),\eps )&=-\frac{\ld \eps^2}{4} \int_{\R} |W|^4 dx -\ld \eps^\frac 32 \Re \int_{\R} |W|^2 W  \overline{ Z(u)} dx  \\
&\hphantom{X}-\frac{\ld \eps}{2} \int_{\R} |W|^2 |Z|^2 dx-\ld \eps \int_{\R}  |\Re (W \overline{Z(u)} ) |^2  dx \notag  \\
&\hphantom{X}- \ld  \eps^\frac 12 \int_{\R} |Z(u)|^2 \Re (W  \overline{ Z(u)} ) dx 
\end{align*}

\noi
and
\begin{align}
\mathcal{E}_2(\mathbb{W}, Z(u),\eps )=A\bigg| \int_{\R} :\! |\eps^\frac 12 W+Z(u) |^2   \! :  dx \bigg|^3 -A\bigg( \int_{\R} |Z(u)|^2  dx \bigg)^3.
\label{AP2}
\end{align}

\end{lemma}

In the variational problem \eqref{APE4}, the well-behaved (coercive) terms  are the following positive terms
\begin{align*}
\mathcal{C}(u):= \E \bigg[A\bigg( \int_{\R} |Z(u)|^2 dx \bigg)^3  +\frac 12 \int_0^1  \| u(t) \|^2_{L^2(\R)} dt  \bigg].
\end{align*}

\noi
By using the coercive structure $\mathcal{C}(u)$, we can control the error term $\mathcal{E}$, which vanishes as  $\eps \to 0$.

\begin{lemma}\label{LEM:error}
Let $\mathcal{E}(\mathbb{W}, Z(u),\eps )$ be as defined in \eqref{AP1}. Then, 
\begin{align*}
|\mathcal{E}(\mathbb{W}, Z(u),\eps )| &\les  \eps^\frac 12 \cdot C+\eps^\frac 12  \cdot \mathcal{C}(u),
\end{align*}

\noi
where $C$ arises from the expected values of the higher moments for each component of $\mathbb{W}=(W, \, :\! |W|^2 \!: )$ in $\mathcal{H}^{-\eta}$. In particular, for any $u\in \Ha$
\begin{align*}
\mathcal{F}^{V+f, \eps}(u) &\ge - \eps^\frac 12 C +(1-\dl)\mathcal{C}(u)\\
\mathcal{F}^{V+f, \eps}(u) &\le  \eps^\frac 12 C +(1+\dl) \mathcal{C}(u)
\end{align*}

\noi
for some small $\dl>0$, independent of $u\in \Ha$, where $C$ depends only on the expected values of higher moments of each component of $\mathbb{W}=(W, \, :\! |W|^2 \!: )$ in $\mathcal{H}^{-\eta}$.  

\end{lemma}

For the proof of Lemma \ref{LEM:error}, we can follow \cite[Lemma 4.1]{OSeoT}.

\begin{ackno}\rm
The work of P.S. is partially supported by NSF grants DMS-1811093 and DMS-2154090. 
\end{ackno}

\end{document}